\documentclass[12pt]{amsart}

\usepackage{amsmath} 
\usepackage{amsfonts}
\usepackage{amssymb}
\usepackage{amsthm}
\usepackage{latexsym}
\usepackage{color,xcolor}
\usepackage{enumerate}
\usepackage{mathrsfs}
\usepackage{setspace}

\newtheorem{thm}{Theorem}[section]
\newtheorem{lemma}[thm]{Lemma}
\newtheorem{prop}[thm]{Proposition}
\newtheorem{coro}[thm]{Corollary}
\newtheorem{remark}[thm]{Remark}
\numberwithin{equation}{section}

\newcommand{\pr}{\partial}
\newcommand{\veps}{\varepsilon}
\newcommand{\dd}[2]{\dfrac{\partial #1}{\partial #2}}

\newcommand{\definedas}{\mathrel{\raise.095ex\hbox{\rm :}\mkern-5.2mu=}}
\newcommand{\asdefined}{\mathrel{=\mkern-5.2mu}\raise.095ex\hbox{\rm :}\;}

\def\tr{\textmd{tr}}

\def\M{\mathscr{M}}
\def\Lap{\Delta} 
\def\grad{\nabla}
\def\Ric{\textmd{Ric}}

\def\dint{\displaystyle\int}
\def\R{\mathbb{R}}

\def\R{\mathbb{R}}

\def\S{\Sigma}
\def\s{\sigma}
\def\({\left(}
\def\){\right)}
\def\a{\alpha}
\def\b{\beta}
\def\={\stackrel{(n=2)}{=}}

\def\bS{\mathbb{S}}
\def\k{\kappa}

\newcommand{\mhh}{\mathfrak{m}^{AH}_H}
\newcommand{\m}{\mathfrak{m}}

\newcommand{\umb}{u_{m,b}}
\newcommand{\ddt}[2]{\dfrac{d #1}{d #2}}

\begin{document}
	
		\title[Asymptotically hyperbolic extensions]{Asymptotically hyperbolic extensions and \\an analogue of the Bartnik mass}
	\author[{Cabrera Pacheco}]{Armando J. {Cabrera Pacheco}}
	\address{Department of Mathematics,  University of Connecticut, Storrs, CT 06269, USA.}
	\curraddr{Department of Mathematics, Universit\"at T\"ubingen,  72076 T\"{u}bingen, Germany.}
	\email{cabrera@math.uni-tuebingen.de}
	
	\author[Cederbaum]{Carla Cederbaum}
	\address{Department of Mathematics, Universit\"at T\"ubingen,  72076 T\"{u}bingen, Germany.}
	\email{cederbaum@math.uni-tuebingen.de}
	
	\author[McCormick]{Stephen McCormick}
	\address{Institutionen f\"{o}r Matematik, Kungliga Tekniska H\"{o}gskolan, 100 44 Stockholm, Sweden.}
	\curraddr{Matematiska institutionen, Uppsala universitet, 751 06 Uppsala, Sweden.}
	\email{stephen.mccormick@math.uu.se}

\begin{abstract}
The Bartnik mass is a quasi-local mass tailored to asymptotically flat Riemannian manifolds with non-negative scalar curvature. From the perspective of general relativity, these model time-symmetric domains obeying the dominant energy condition without a cosmological constant. There is a natural analogue of the Bartnik mass for asymptotically hyperbolic Riemannian manifolds with a negative lower bound on scalar curvature which model time-symmetric domains obeying the dominant energy condition in the presence of a negative cosmological constant.
				
Following the ideas of Mantoulidis and Schoen \cite{M-S}, of Miao and Xie \cite{M-X}, and of joint work of Miao and the authors \cite{CCMM}, we construct asymptotically hyperbolic extensions of minimal and constant mean curvature (CMC) Bartnik data while controlling the total mass of the extensions. We establish that for minimal surfaces satisfying a stability condition, the Bartnik mass is bounded above by the conjectured lower bound coming from the asymptotically hyperbolic Riemannian Penrose inequality. We also obtain estimates for such a hyperbolic Bartnik mass of CMC surfaces with positive Gaussian curvature.
\end{abstract}

\maketitle
\section{Introduction}  \label{SIntro}
In a recent paper of Miao and the authors \cite{CCMM}, we constructed extensions of constant mean curvature (CMC) Bartnik data with controlled ADM mass \cite{ADM} in order to estimate Bartnik's quasi-local mass \cite{Bartnik-89}. By \emph{Bartnik data}, we mean a triple $(\S,g,H)$ consisting of a metric $g$ on a surface $\S\cong \bS^2$ and a non-negative function $H$ (constant and positive in the CMC case) on $\S$. An admissible extension in the context of the Bartnik mass is then an asymptotically flat Riemannian manifold with non-negative scalar curvature whose boundary is isometric to $(\S,g)$ with induced mean curvature $H$, and with no closed minimal surfaces enclosing $\S$. The estimates in \cite{CCMM} are obtained by constructing examples of admissible extensions with controlled mass; a construction based on the ideas of Mantoulidis and Schoen \cite{M-S}, who recently proved that the Bartnik mass of stable minimal surface is exactly its Hawking mass, and on a collar construction by Miao and Xie \cite{M-X} in which the growth of the Hawking mass along its level sets is well-controlled. In this article, we give an analogue of both of these results in the asymptotically hyperbolic case.

There is a natural analogue of the Bartnik mass for asymptotically hyperbolic manifolds, which to the best of our knowledge, has received little attention in the literature to date. We would like to remark that such a quantity motivates the work of Bonini and Qing~\cite{BQ} and ongoing work of Martin \cite{M}, but as we are unaware of a precise definition in the literature, we give such a definition in Section \ref{Sdefn}. 

Here we give estimates for this asymptotically hyperbolic analogue of the Bartnik mass. The idea is to construct a ``collar" manifold with two boundary components; one realising the Bartnik data, while the other is a round sphere with controlled hyperbolic Hawking mass (see \eqref{eq-AHHdef} below). We then smoothly glue this collar to an AdS-Schwarzschild manifold with mass close to the Hawking mass of the end of the collar. 

We will give estimates in two cases: First, we discuss the case $H\equiv H_{o}=0$ (when the Bartnik data corresponds to a minimal surface). Then we provide estimates for Bartnik data when $H\equiv H_o$ is a positive constant, which we refer to as \emph{CMC Bartnik data}. The main results are indicated in Theorem \ref{main1} (minimal surfaces) and Theorem \ref{main2} (CMC Bartnik data). The precise statements are given as Theorem \ref{thm-main} and Theorems \ref{thm-extension1} and \ref{thm-extension2}, respectively. All necessary definitions will be given in Section \ref{Sdefn}.

\begin{thm}\label{main1}
Let $(\S \cong \bS^2,g_o,H_o=0)$ be Bartnik data satisfying $K(g_{o})>-3$ or satisfying $\lambda_1(-\Delta_{g_o}+K(g_o))>0$, where $\lambda_1(-\Delta_{g_o}+K(g_o))$ denotes the first eigenvalue of the operator $-\Delta_{g_o}+K(g_o)$ on $\S$ and $K(g_o)$ denotes the Gaussian curvature of $g_o$. Then its hyperbolic Bartnik mass satisfies
\begin{align*}
\m_{B}^{AH}(\S,g_o,H_{o}=0)\leq\sqrt{\frac{|\S|_{g_o}}{16\pi}}\(1+\frac{|\S|_{g_o}}{4\pi}\).
\end{align*}
\end{thm}

\begin{remark}
The quantity on the right-hand side of this inequality
is precisely the bound in the conjectured Riemannian Penrose inequality in the asymptotically hyperbolic case. This conjecture states that for an asymptotically hyperbolic manifold with scalar curvature $R \geq -6$, with outermost minimal surface $\S$, we have
\begin{align*}
m\geq \sqrt{\frac{|\S|_{g_o}}{16\pi}}\(1+\frac{|\S|_{g_o}}{4\pi}\),
\end{align*}
where $m$ is the hyperbolic mass of the manifold (see Section \ref{SS-mass}). The interested reader is referred to the review article by Mars \cite{Mars}. This inequality has since been proven in some cases (see for example \cite{DGS}, \cite{dLG} and \cite{A}), however the general case remains open. If the general case is established, Theorem \ref{main1} then implies the equality
\begin{align*}
\m_B^{AH}(\S,g_o,H_{o}=0)=\sqrt{\frac{|\S|_{g_o}}{16\pi}}\(1+\frac{|\S|_{g_o}}{4\pi}\),
\end{align*}
in analogy with the asymptotically flat case established in \cite{M-S}.
\end{remark}

\begin{thm}\label{main2}
Let $(\S \cong \bS^2,g_o,H_o)$ be Bartnik data with Gaussian curvature $K(g_o)>-3$ and $H_o$ a positive constant. Assume that its hyperbolic Hawking mass satisfies
\begin{align*}
\mhh(\S,g_o,H_o)>-\(\frac{|\S|_{g_o}}{4\pi} \)^{\frac{3}{2}}.
\end{align*}
Then its hyperbolic Bartnik mass satisfies 
\begin{align*}
\m_B^{AH}(\S,g_o,H_o) \leq \mhh(\S,g_o,H_o)+\textnormal{Err}(H_o\sqrt{\alpha}),
\end{align*}
where $\textnormal{Err}(y)$ approaches zero as $y$ goes to zero, and $\alpha$ approaches zero as $g_o$ tends to a round metric. In particular, the upper bound for the hyperbolic Bartnik mass approaches the hyperbolic Hawking mass as $H_o$ goes to zero or as $g_o$ tends to a round metric.
\end{thm}
The precise form of the error term $\textnormal{Err}(H_o\sqrt{\alpha})$ appearing in Theorem \ref{main2} can be found in Theorems \ref{thm-extension1} and \ref{thm-extension2}.

	The outline of this paper is as follows. In Section \ref{Sgluing}, we provide a tool to smoothly glue together two spherically symmetric Riemannian manifolds while keeping the scalar curvature bounded from below. In Section \ref{SHzero}, we consider the case where the Bartnik data corresponds to a stable minimal surface (apparent horizon) and prove Theorem \ref{thm-main}.
	
	In Section \ref{SHnonzero}, we give a collar construction motivated by that of Miao and Xie~\cite{M-X} and study its behaviour with respect to the variation of several parameters. Then we smoothly glue it to a spatial AdS-Schwarzschild manifold to prove Theorems \ref{thm-extension1} and \ref{thm-extension2}.
	
	It should be remarked that, while we consider only $2$-dimensional Bartnik data ($3$-dimen\-sional asymptotically hyperbolic manifolds) here, we do not rely on the dimension in any critical way. Miao and the first-named author have shown that the construction method of asymptotically flat extensions in the work of Mantoulidis and Schoen \cite{M-S} can be extended to higher dimensions in \cite{CM}. Provided one can obtain the corresponding smooth path of metrics needed for the collar extension, following this work would naturally lead to a higher dimensional analogue of the work considered here. However, for the sake of exposition, we do not pursue any higher dimensional result here, except in Section \ref{Sgluing} where the proofs are nearly identical in higher dimensions and no additional definitions nor concepts are needed.
	
\bigskip

{\bf Acknowledgements} 
All three authors thank the Erwin Schr\"{o}dinger Institute for hospitality and support during our visits in 2017 in the context of the program \emph{Geometry and Relativity}. AJCP and CC thank the Carl Zeiss foundation for generous support. The work of CC and SM was partially supported by the DAAD and Universities Australia. The work of AJCP was partially supported by the NSF grant DMS 1452477. The work of CC is supported by the Institutional Strategy of the University of T\"ubingen (Deutsche Forschungsgemeinschaft, ZUK 63). SM is grateful for support from the Knut and Alice Wallenberg Foundation.

\section{Definitions and Facts}\label{Sdefn}
From the perspective of general relativity, an asymptotically hyperbolic Riemannian $3$-manifold represents time-symmetric initial data for a gravitating system in the presence of a (negative) cosmological constant. Generally, the cosmological constant is set to $-3$, in which case time-symmetric initial data satisfying the dominant energy condition corresponds to asymptotically hyperbolic Riemannian $3$-manifolds with scalar curvature bounded below by $-6$. One may view the cosmological constant as a (negative) vacuum energy density. Therefore, if one seeks to measure the quasi-local mass of a bounded domain $\Omega$ in such a manifold, the vacuum energy density should somehow be compensated for. As an important example of such a compensation, one may compare the Hawking mass~\cite{Hawking}
	\begin{align*}
	\m_H(\S,g,H)\definedas\sqrt{\frac{|\S|_g}{16\pi}}\( 1-\frac{1}{16\pi}\int_\S H^2\,d\sigma \) 
	\end{align*}
	with the following hyperbolic version that is known in the literature
	\begin{align}\label{eq-AHHdef}
	\mhh(\S,g,H)\definedas\sqrt{\frac{|\S|_g}{16\pi}}\( 1-\frac{1}{16\pi}\int_\S \left(H^2-4\right)d\sigma \).
	\end{align}
		
\subsection{Hyperbolic Bartnik mass}\label{SS-mass}
The natural analogue of the usual Bartnik mass of Bartnik data $(\S,g,H)$ is given by the infimum, over a space of ``admissible asymptotically hyperbolic extensions" $(M,\gamma)$ of $(\S,g,H)$, of the well-known total hyperbolic mass -- see below. We briefly describe the notion of asymptotically hyperbolic Riemannian $3$-manifolds and the mass of such a manifold.
	
A Riemannian $3$-manifold $(M,\gamma)$ is called \emph{asymptotically hyperbolic} if there exists a compact set $K \subset M$, a closed ball $B\subset\mathbb{H}^3$, and a diffeomorphism from $M\setminus K$ to $\mathbb{H}^{3}\setminus B$ such that the push-forward of the metric $\gamma$ can be written as the hyperbolic metric plus a well-controlled error term. Then, the \emph{(total) hyperbolic mass} $m(M,\gamma)$ is given by a quantity computed from this error term which turns out to be an invariant under change of diffeomorphism. For precise definitions the reader is referred to X.~Wang \cite{WangMass}; see also Chru\'sciel and Herzlich \cite{ChruscielHerzlich}. Since all the extensions considered here will be exactly 3-dimensional spatial AdS-Schwarzschild manifolds $((r_+,\infty)\times\mathbb{S}^2,g_{m,1})$ (see below) outside a compact set (which are well-known to be asymptotically hyperbolic of mass $m$), we do not go into any detail here.

Now let $\mathcal{A}(\S,g,H)$ be the set of \emph{admissible extensions $(M,\gamma)$}; namely, of smooth asymptotically hyperbolic $3$-manifolds $(M,\gamma)$ with scalar curvature $R(\gamma)\geq -6$, containing no minimal surfaces enclosing the boundary (except possibly the boundary itself), and whose boundary is isometric to $(\S,g)$ with induced mean curvature $H$. We then define the \emph{hyperbolic Bartnik mass}
\begin{align*}
\m_{B}^{AH}(\S,g,H)\definedas\inf\{m(M,\gamma): (M,\gamma)\in \mathcal{A}(\S,g,H) \}, 
\end{align*}
where $m(M,\gamma)$ denotes the total hyperbolic mass of $(M,\gamma)$.

\begin{remark}
The condition that an admissible extension must not contain minimal surfaces enclosing the boundary is introduced to rule out extensions where the Bartnik data is hidden behind a horizon (minimal surface). As Bartnik pointed out in his original definition in the asymptotically flat case, extensions like this could have arbitrarily small (positive) mass, rendering the definition somewhat meaningless.
		
It is worth remarking that if the boundary is outer minimising then, in the asymptotically flat case, it is known that the Hawking mass of the boundary provides a lower bound for the ADM mass via Huisken and Ilmanen's proof of the Riemannian Penrose inequality \cite{H-I01}. For this reason, sometimes the class of admissible asymptotically flat extensions is further restricted to manifolds where the boundary is outer-minimising. It is conjectured that this restriction does not change the infimum.
\end{remark}	

\begin{remark}
Bartnik's original definition of quasi-local mass assigned a mass to a domain $\Omega$ with compact boundary contained in an asymptotically flat manifold -- as opposed to assigning the mass to Bartnik data, as described above. The extensions that he considered were asymptotically flat manifolds in which this bounded domain could be isometrically embedded. However, if a minimising extension exists then it will generically fail to be smooth at $\partial\Omega$, which leads one to consider extensions that are not smooth at $\partial\Omega$ provided that an appropriate version of the positive mass theorem holds for such a non-smooth manifold. In this case the appropriate positive mass theorem for such a manifold with corner along a hypersurface was established in the asymptotically flat case independently by Miao \cite{Miao02}, and by Shi and Tam~\cite{ShiTam02}. More recently, Bonini and Qing \cite{BQ} have proven a positive mass theorem for asymptotically hyperbolic manifolds with corners.

Therefore, one often considers ``extensions'' to be asymptotically flat manifolds with boundary satisfying certain geometric boundary conditions \cite{Bartnik-TsingHua}. The Bartnik mass is then considered to be a quantity associated with Bartnik data, as described in Section \ref{SIntro}, both in the asymptotically flat and the asymptotically hyperbolic setting.
\end{remark}

\subsection{Spatial AdS-Schwarzschild manifolds}\label{SS-adsschwarzschild}
	
	Arguably the most important example of an asymptotically hyperbolic manifold is the spatial AdS-Schwarzschild manifold. This is the asymptotically hyperbolic analogue of the standard spatial Schwarzschild manifold, and represents canonical initial data for a static, spherically symmetric black hole sitting in the presence of a negative cosmological constant, and an otherwise vacuum spacetime. For later convenience, we will introduce a larger class of Riemannian $3$-manifolds with two parameters, $m\in\R$  and $b > 0$. The AdS-Schwarz\-schild metric will arise as a special case.
	
Now consider the family of metrics on $(r_+,\infty) \times \bS^2$ given by
	\begin{align}\label{eq-AdS-Schwarzschild-metric}
	g_{m,b}=\(1+br^2-\frac{2m}{r}\)^{-1}dr^2+r^2 g_*,
	\end{align}
	where $r_+=r_+(m,b)$ is the largest root of $1+br^2-\frac{2m}{r}$ if $m>0$ and otherwise $r_+=0$. Here $g_*$ denotes the standard round metric on $\bS^2$ with area $4\pi$.
		
	The metric $g_{m,b}$ is a static solution to the vacuum Einstein constraint equations in the presence of a cosmological constant $\Lambda=-3b$ or ``radius'' $1/\sqrt{b}$ when considered as asymptotic to a hyperboloid. In particular, it has scalar curvature $R(g_{m,b})=-6b$ and by special choices of the parameters $m$ and $b$, we recover several well-known manifolds: When $m=0$ and $b>0$, we recover a hyperbolic space of radius $\frac{1}{\sqrt b}$; when $b=0$ and $m\neq0$ we recover the spatial Schwarzschild metric of mass $m$; when $b=1$ and $m\neq0$ we recover the AdS-Schwarzschild metric of mass $m$. We remark that the hyperbolic Hawking mass \eqref{eq-AHHdef} is constant along the centred spheres in the spatial AdS-Schwarzschild manifold just as the usual Hawking mass is constant along the centred spheres in the spatial Schwarzschild manifold.

We will now discuss further properties of the metrics $g_{m,b}$, which will be used later in the collar extensions constructed in Section \ref{SHnonzero}. For any $r_o \geq r_+$, define
\begin{align*}
s(r) = \dint_{r_o}^r \( 1+b t^2 - \frac{2m}{t} \)^{-1/2} \, dt
\end{align*}
on $[r_o,\infty)$. Then we can write the metric $g_{m,b}$ as
\begin{align}\label{sch-um}
g_{m,b}= ds^2 + u_{m,b}(s)^2 g_*,
\end{align}
where $u_{m,b}$, defined on $[0,\infty)$, is the inverse of $s=s(r)$ on $[r_o,\infty)$. 
\newpage

One can directly check that the function $u_{m,b}$ satisfies the following properties:\\[-1.5ex]
	\begin{enumerate}[(a)]\itemsep1.5ex 
		\item $u_{m,b}(0)=r_o$, 
		\item $u_{m,b}' = \sqrt{1+bu_{m,b}^2- \frac{2m}{u_{m,b}}}$, and
		\item $u_{m,b}'' = b u_{m,b}+ \dfrac{m}{u_{m,b}^2}$.\\[-1.5ex]
	\end{enumerate}
	In addition, combining (a), (b), and (c) above, one can directly check that $u_{m,b}$ satisfies
	\begin{align}
	1+3bu_{m,b}^2 - u_{m,b}'^2 - 2 u_{m,b} u_{m,b}''=0,
	\end{align}
	which is equivalent to the condition $R(g_{m,b})= - 6b$, as we will see below.

\subsection{Collar extensions}\label{sec:computations}
The asymptotically flat extensions constructed in \cite{M-S} and \cite{CCMM} provide a way to construct admissible extensions (in the sense of the Bartnik mass) with control on the ADM mass for minimal and CMC Bartnik data, respectively. A key step in the construction in the minimal case is to obtain a collar extension $(M,\gamma)$ of a given $(\S,g_{o})$ in such a way that its initial boundary is isometric to $(\S,g_{o})$ and minimal, while controlling the growth of the area along the collar. For Bartnik data $(\S,g_{o},H_{o})$ with constant $H_{o} > 0$, one is more generally interested in controlling the growth of the Hawking mass along the collar. For our construction, we will apply the general idea first put forward in \cite{M-S}, using appropriate collar extensions as in \cite{CCMM} (see also \cite{M-X}) in the CMC case.
	
One main ingredient of the construction of an extension \`a la Man\-tou\-lidis and Schoen \cite{M-S} is the existence of a path of $2$-metrics connecting a given $2$-metric $g_o$ to a round metric $r_o^{2}\,g_{*}$ of the same area on $\S\cong \bS^{2}$ and using it to construct a collar extension realising the given Bartnik data as the inner boundary of the collar. More concretely, for a $2$-metric $g_o$ (obeying certain curvature conditions), they prove existence of a path $\{g(t)\}$ connecting $g(0)=g_o$ to a round metric $g(1)=r_o^{2}\,g_{*}$ while preserving the curvature conditions such that\\[-1.5ex]
\begin{enumerate}\itemsep1.5ex
\item $g'(t) = 0$ on $[\theta,1]$ for some $0<\theta<1$, and
\item $\tr_{g(t)} g'(t)=0$ on $[0,1]$,\\[-1.5ex]
\end{enumerate}
where the latter condition implies that the area form is preserved along the path. Existence of such a path under the curvature conditions we will impose is obtained in the minimal surface case using the Uniformisation Theorem as in \cite{M-S} (see Section \ref{SHzero}) and in the CMC case using Ricci flow (see Section \ref{SHnonzero}).
	
	A \emph{collar extension} is then given by the manifold $[0,1] \times \S$ with metric $\gamma$ 
	\begin{align}\label{collar-general}
	\gamma \definedas v(t,\cdot)^2 dt^2 + E(t)^2 g(t),
	\end{align}
	where $v$ and $E$ are smooth positive functions, and $E$ satisfies $E'>0$ and $E(0)=1$. Throughout this work, we will use the notation $\S_{t}\definedas\{t\}\times\S$ to denote the $t$-level set of $[0,1]\times\S$. Here, $\gamma$ induces the metric $g_o$ on $\S_0$ and induces a round metric on $\S_t$ for all $t \in [\theta,1]$. For simplicity we avoid writing explicitly the dependence on $x\in\S$ of various quantities.
	
	A direct computation shows that the scalar curvature of the collar extension is given by
	\begin{align}\label{scalar-collar-general}
	\begin{split}
		R( \gamma) &= 2v(t,\cdot)^{ -1} \( - \Lap_{E(t)^2g(t)} v(t,\cdot) + \dfrac{1}{2} R(E(t)^{2}g(t)) v(t,\cdot)  \)\\
		&\quad + v(t,\cdot)^{ -2}\left[  \dfrac{-2 E'(t)^2 -4  E(t)E''(t)}{E(t)^2} - \dfrac{1}{4} |g'(t)|^2_{g(t) }  + 4 \dfrac{\pr_t v(t,\cdot)}{v(t,\cdot)} \dfrac{E'(t)}{E(t)}\right],
	\end{split}
	\end{align}
(see \cite{M-S,M-X}). The mean curvature of the $t$-level set $\S_t$ along the collar extension is given by
	\begin{align}\label{eq-Hs}
	H(t) = \dfrac{2 E'(t)}{v(t,\cdot) E(t)}.
	\end{align}
We will be particularly interested in the hyperbolic Hawking mass of the $t$-level sets in the case that $v$ is constant on  each $\S_t$, which in that case is given by
	\begin{align} \label{Hawking-level}
	\mhh(\S_t ) = \dfrac{E(t) r_o}{2}\(1-  \dfrac{r_o^2 E'(t)^2}{v(t,\cdot)^2}  + r_o^2 E(t)^2 \).
	\end{align}
	
	Expressions \eqref{Hawking-level} and \eqref{scalar-collar-general} suggest that by choosing carefully the functions $E$ and $v$, one can control the growth of the hyperbolic Hawking mass of the $t$-level sets of the collar extension as well as its scalar curvature. We also note that, as in \cite{CCMM}, condition (1) is not necessary as one can approximate $\{ g(t) \}$ by a family $\{ g_{\theta}(t) \}$ of paths with $\theta \to 1$. However, since the the proof of our main result is constructive, we keep imposing this condition. In the following section we provide tools to smoothly glue a spatial AdS-Schwarzschild manifold to a certain type of collar extension, resulting in admissible extensions of the given data $(\S,g_o,H_{o})$. We remark that the hyperbolic mass of such an extension provides an upper bound for  $\m_{B}^{AH}(\S,g_o,H_{o})$.

	\section{Gluing via an $R>\tau$ bridge in spherical symmetry}\label{Sgluing}
	The main results of this paper are obtained by constructing a collar manifold with one boundary component realising the given Bartnik data and then gluing the collar to an exterior AdS-Schwarzschild manifold via an $R > -6$ bridge. In this section, we provide the gluing tools to be used in Sections \ref{SHzero} and \ref{SHnonzero}. The following lemma is a natural generalisation of Lemma 2.2 in \cite{M-S} (see also Lemma 2.1 of \cite{CCMM}), and the proof closely follows the original one. We state and prove it in $n$ dimensions as it may be of interest for other applications, although we will only apply it for $n=2$.
	\begin{lemma}[Smooth gluing lemma]\label{lemma-gluing}
		Let $f_i:[a_i,b_i]\rightarrow\mathbb{R}^+, i=1,2$, be two smooth positive functions, let $g_*$ be the standard metric on $\bS^n$ ($n\geq2$), and let $\tau\in(-\infty,0]$ be some constant.
		Suppose that\\[-1.5ex]
		\begin{enumerate}[(i)]\itemsep1.5ex
			\item the metric $\gamma_i\definedas ds^2+f_i(s)^2g_*$ has scalar curvature $R(\gamma_i)>\tau$,
			\item $f_1(b_1)<f_2(a_2)$,
			\item $0<f'_1(b_1)<\sqrt{1-\frac{\tau}{n(n-1)}f_1(b_1)^2},$ and
			\item $-\sqrt{1-\frac{\tau}{n(n-1)}f_2(a_2)^2}<f'_2(a_2)\leq f'_1(b_1)$.\\
		\end{enumerate}
		Then, after translating $[a_i,b_i]$ appropriately, there is a smooth positive function $f$ on $[a_1,b_2]$ such that\\[-1.5ex]
		\begin{enumerate}[(I)]\itemsep1.5ex
			\item $f \equiv f_1$ on $[a_1,\frac{a_1+b_1}{2}]$ and $f \equiv f_2$ on $[\frac{a_2+b_2}{2},b_2]$ and
		\item the metric $\gamma\definedas dt^2+f(s)^2g_*$ has scalar curvature $R(\gamma_b)>\tau$ on $[a_1,b_2]\times \bS^n$.\\[-1.5ex]
		\end{enumerate}
		
		In addition, if $f_i' >0$ on $[a_i,b_i]$ then $([a_1,b_2]\times \bS^n,\gamma)$ is foliated by mean convex CMC spheres.
	\end{lemma}
	\begin{proof}
		First note that it is possible to translate $[a_i,b_i]$ to ensure that $b_1<a_2$ as well as the existence of a function $\zeta\in C^1[b_1,a_2]$ satisfying (see \cite[Lemma 2.1]{CCMM})\\[-1.5ex]
		\begin{itemize}\itemsep1.5ex
			\item $\zeta(b_1)=f'_1(b_1)$,
			\item $\zeta(a_2)=f'_2(a_2)$,
			\item $\zeta'\leq 0$ on $[b_{1},a_{2}]$, and
			\item $ \int_{b_1}^{a_2}\zeta(t)\,dt=f_2(a_2)-f_1(b_1)>0$.\\[-1.5ex]
		\end{itemize}
	Using $\zeta$, we then define the function 
		\begin{align*}
		\widehat f (s) \definedas f_1(b_1) + \int_{b_1}^s  \zeta(t)\, dt,
		\end{align*}
		which clearly satisfies\\[-1.5ex]
		\begin{itemize}\itemsep1.5ex
			\item $ \widehat{f} (b_1)= f_1(b_1)$ and $ \widehat f (a_2)= f_2(a_2)$,
			\item $ \widehat{f} ' (b_1)= f_1'(b_1)$ and $ \widehat f' (a_2) = f_2'(a_2)$,
			\item $  f_1'(b_1 ) \geq \widehat{f}' \geq f_2'(a_2) $ on $(b_1,a_2)$, and
			\item  $  \widehat{f} '' =  \zeta' \le 0  $ on $[b_1,a_2]$.\\[-1.5ex]
		\end{itemize}
		From the condition $\widehat{f}''\leq0$ and $f'_1(b_1)>0$, we have $\widehat{f}>\widehat{f}(b_1)$ on $(b_1,a_2]$. On $[a_1, b_2]$, define 
		\begin{align}\label{eq:defwidetildef}
		\widetilde{f} \definedas \begin{cases}
		f_1&\text{ on } [a_1,b_1]\\
		\widehat{f} &\text{ on } (b_1,a_2)\\
		f_2&\text{ on } [a_2,b_2]
		\end{cases} . 
		\end{align}
		Then clearly $\widetilde{f}\in C^{1,1}([a_1, b_2]) $, $\widetilde{f}$ is $C^2$ away from $b_1,a_2$ and $ \widetilde f > 0 $. We now smooth out $\widetilde{f}$ using an appropriate mollification (as in \cite{CM,CCMM}):
		
		Let $\delta >0$ satisfy
		\begin{align*}
		\dfrac{a_1+b_1}{2} < b_1 - \delta\ \ \text{and} \ \ a_2+\delta < \dfrac{a_2+b_2}{2}. 
		\end{align*}
		Now let $\eta_\delta$ be a smooth cut-off function that equals $1$ on $[b_1-\delta,a_2+\delta]$, vanishes on the set $\left[a_1,\frac{a_1+b_1}{2} \right]\cup \left[\frac{a_2+b_2}{2},b_2\right]$, and satisfies $ 0 <  \eta_{\delta}   < 1 $ elsewhere. Let $\phi:\mathbb{R}\to\mathbb [0,\infty)$ be a standard smooth mollifier with compact support in $[-1,1]$ and $\int_{-\infty}^\infty\phi(t)\,dt=1$.
		
		Given any $\varepsilon\in(0,\frac{\delta}{4})  $, we define $f_{\varepsilon}$ by 
		\begin{align}\label{eq-def-fe}
		f_{\varepsilon}(s)\definedas \dint_{-\infty}^\infty \widetilde{f}(s-\varepsilon \eta_{\delta}(s)t)\phi(t)\, dt \quad \text{ for }\quad  s \in [a_1, b_2].  
		\end{align}
		Note that $f_{\varepsilon}$ is smooth on $[a_1,b_2]$, $f_\varepsilon\equiv \widetilde{f}$ on $\left[a_1,\frac{a_1+b_1}{2} \right]\cup \left[\frac{a_2+b_2}{2},b_2\right]$ and 
		\begin{align}\label{eq-def-fe'}
		f'_{\varepsilon}(s)=\dint_{-\infty}^\infty \widetilde{f}' (s-\varepsilon \eta_{\delta}(s)t)(1 - \varepsilon \eta_\delta'(s) t  ) \phi(t)  \,  dt  \quad \forall\, s \in [a_1, b_2] .
		\end{align}
		As $ \widetilde{f}'$ is $C^0$ everywhere and $C^1$ except possibly at $b_1$ and $a_2$, by standard mollification arguments we have for $s \in (b_{1}-\delta,a_{2}+\delta)$
		\begin{align}\label{eq-def-fe''-1}
		f''_{\varepsilon}(s) = & \  \frac{d}{ds} \left(  \dint_{-\infty}^\infty \widetilde{f}' (s-\varepsilon t) \phi(t)  \,  dt \right) 
		=  \   \dint_{-\infty}^\infty \widetilde{f}''  (t) \phi_\varepsilon ( s -t)   \,  dt, 
		\end{align}
		using $\eta_{\delta}(s)=1$, where $ \phi_\varepsilon(s) \definedas \frac{1}{\varepsilon} \phi ( \frac{ s}{\varepsilon} )$. Moreover, for $s \in [a_{1},b_{2}]\setminus[b_1 -  \frac{1}{4} {\delta}, a_2 + \frac{1}{4}  \delta] $ where $\widetilde{f}$ is smooth, we have
		\begin{align}\label{eq-def-fe''-2} 
		\begin{split}
		f''_{\varepsilon}(s) = & \ \int_{-\infty}^\infty \widetilde{f}'' (s-\varepsilon \eta_{\delta}(s)t)\,(1 - \varepsilon \eta_\delta'(s) t  )^2 \phi(t)  \,  dt  \\ 
		& \ - \varepsilon \int_{-\infty}^\infty \widetilde{f}' (s-\varepsilon \eta_{\delta}(s)t)\,  \eta_\delta'' (s) t   \phi(t)  \,  dt.
		\end{split} 
		\end{align}
		
		We will now show that for sufficiently small $\varepsilon>0$,  the metric $\gamma_{\varepsilon}\definedas d s^2 + f_\varepsilon (s)^2  g_*$ has scalar curvature $R(\gamma_{\varepsilon})>\tau$. Given a metric of the form $\gamma\definedas ds^2+f(s)^{2}g_*$ with smooth positive $f$, a direct computation shows that the scalar curvature condition $R(\gamma)>\tau$ for some given $\tau\in\mathbb{R}$ is equivalent to (see \cite{CM} Eq. (4.13))
		\begin{align}\label{eq-fRcond}
		f''<\frac{(n-1)}{2f}\left(1-(f')^2-\frac{\tau f^2}{n(n-1)}\right).
		\end{align}
		Now we define the quantity
		\begin{align*}
		\Omega[f]\definedas \frac{(n-1)}{2f}\( 1-(f')^2-\frac{\tau f^2}{n(n-1)} \)
		\end{align*}
		as a shorthand for the right-hand side of \eqref{eq-fRcond}. In analogy to the asymptotically flat case, the key to this proof will be the fact that
		\begin{align*}
		\Omega[\widetilde{f}\,]-\widetilde f''>3d>0
		\end{align*}
		holds for the function $\widetilde f$ defined by \eqref{eq:defwidetildef} and some number $d>0$. To see this, recall first that $\Omega[\widetilde f\,]\in C^{0,1} ([a_1, b_2])$. Since $\widehat{f}''\leq0$ on $[b_1,a_2]$, $\widetilde{f}$ attains its maximum at some $s_*\in[b_1,a_2]$ and thus  $\widehat{f}'(s_*)=0$. Moreover $\widehat{f}' \geq 0$ on $[b_1,s_*]$ and $\widehat{f}' < 0$ on $(s_*,a_2]$. If $s_* = a_2$, since $\widehat{f}' \leq f_1'(b_1)$ and (iii) holds we know that
		\begin{align*}
		0 \leq \widehat{f} < \sqrt{1 - \dfrac{\tau}{n(n-1)} \widehat{f}\;},
		\end{align*}
		and hence $\Omega[f] > 0$ on $[b_1,a_2]$. If $s_* < a_2$, by the previous argument $\Omega[f] > 0$ on $[b_1,s_*]$ and since $\widehat{f}$ is strictly decreasing on $(s_*,a_2]$ we have $\widehat{f} > f(a_2)$, which implies
		\begin{align*}
		-\sqrt{1 - \dfrac{\tau}{n(n-1)} f_2(a_2)^2} > -\sqrt{1 - \dfrac{\tau}{n(n-1)}\widehat{f}\;}.
		\end{align*}
		Combining this with (iv) and the fact that $\widehat{f}' \geq f_2'(a_2)$, we also obtain $\Omega[\widehat{f}] > 0$ on $[b_1,a_2]$. Now, as $ \Omega [ \widetilde f \,]=\Omega[\widehat{f}\,]  > 0 $ on $[b_1, a_2]$ and $\widetilde f'' =\widehat{f}''\le 0 $  on $(b_1, a_2)$ by construction of $\widetilde f$, and as $\gamma_i$ has scalar curvature $R(\gamma_{i})>\tau$ on $[a_i, b_i] \times {\bS}^{n}$ for $i=1,2$, we have
		\begin{align}\label{eq:Omega}
		\Omega [ \widetilde f\, ]  -\widetilde f''>0
		\end{align}
		on $ [a_1, b_2] \setminus \{ b_1, a_2 \} $. In fact, $\Omega[\widetilde f\,]-\widetilde f''$ can be bounded from below by
		\begin{align}\label{eq-inf-p}
		\inf_{ s \in  [a_1, b_2] \setminus \{ b_1, a_2 \} } \left( \Omega [ \widetilde f\, ](s)  - \widetilde f''(s) \right) \asdefined 3d> 0,
		\end{align}
		as the inequality \eqref{eq:Omega} holds separately on the compact intervals $[a_{i},b_{i}]$ and $[b_{1},a_{2}]$ and because $\widetilde f$ is $C^{2}$ on these intervals, separately. It follows that
		\begin{align*}
		\widetilde{f}''  \leq\Omega[\widetilde{f}\, ]  -3d
		\end{align*}
		everywhere where $\widetilde f''$ is defined. \\
		
		Now we consider $ \Omega [ f_\varepsilon ]$ in order to show that \eqref{eq:Omega} also holds for $f_{\varepsilon}$ so that $R(\gamma_{\varepsilon})>\tau$. By \eqref{eq-def-fe} and \eqref{eq-def-fe'}, we have $ f_\varepsilon \to \widetilde f $ in $ C^1([a_1, b_2])$ as $\varepsilon\to0^{+}$. Hence, $ \Omega[f_\varepsilon] \to \Omega[\widetilde{f}\,]$ in $ C^0 ([ a_1, b_2])$, 
		which shows, for sufficiently small $ \varepsilon$, 
		\begin{align*}
		\sup_{ s \in [a_1, b_2] } \left| \Omega[\widetilde{f}\,] (s) -\Omega[f_\varepsilon] (s) \right|<d.
		\end{align*} 
		Furthermore, by \eqref{eq-def-fe''-1} and \eqref{eq-def-fe''-2}, and from $f_{\varepsilon}\to\widetilde f$ in $C^1([a_{1},b_{2}])$, we find
		\begin{align*}
		f''_\varepsilon(s)<\sup_{ | t - s | < \varepsilon }\widetilde f''(t)+d \quad \forall\,s \in[a_{1},b_{2}],
		\end{align*}
		for sufficiently small $ \varepsilon>0 $.
		It then follows that, for all $t\in[a_{1},b_{2}]$,
		\begin{align*}
		f''_\varepsilon(s)
		&\leq \sup_{ | t - s | < \varepsilon, \,t\neq b_{1},a_{2} } \left( \Omega[\widetilde{f}\,](t)-3d \right) +d\\
		&<\Omega[\widetilde{f}\,](s)-d\\
		&<\Omega[f_\varepsilon](s);
		\end{align*}
		noting that, in the second-to-last inequality, we also used the uniform continuity of $\Omega[\widetilde{f}\,]$ on $[a_1, b_2]$, and hence $ \Omega[\widetilde{f}\,](t)  < \Omega[\widetilde f\,] (s) + d $ for any $ t $ with $ | t - s | < \varepsilon $  provided $ \varepsilon $ is small enough.
		
		We have therefore shown that \eqref{eq:Omega} and thus \eqref{eq-fRcond} holds on $[a_1, b_2]$ with $f$ replaced by $ f_\varepsilon$ for small $\varepsilon $ on $[a_1, b_2] $. That is, the metric $\gamma_{\varepsilon}\definedas d t^2 + f_\varepsilon(s) ^2 g_*$ has scalar curvature $R(\gamma_{\varepsilon})>\tau$ and thus $f\definedas  f_{\varepsilon}$ can act as the sought after profile function $f$ satisfying the conclusions of the lemma for sufficiently small  $\varepsilon>0$.\\ 
		
		In addition, if $ f_i' > 0 $ then $ \widetilde f' > 0 $, which directly implies $ f'=f_\varepsilon' > 0 $ as $ f_\varepsilon \to \widetilde f $ in $C^1([a_1, b_2])$, proving that the resulting warped product is foliated by mean convex CMC spheres.
	\end{proof}
	
	In order to apply Lemma \ref{lemma-gluing} to  smoothly glue an AdS-Schwarzschild manifold to the collars that we will construct in Section \ref{SHzero} ($H=0$ case) and Section \ref{SHnonzero} ($H>0$ case), we need to slightly bend a piece of the AdS-Schwarzschild metric to increase the scalar curvature in order to make room for smoothing out the gluing zone. This is the analogue of Lemma 2.3 in \cite{M-S}, and the proof is very similar.
	
\begin{lemma}[Bending Lemma]\label{lemma-bending}
Let $\gamma=ds^2+u(s)^2g_*$ be a metric on a cylinder $[a,\infty)\times{\bS^n}$, where $u$ is a smooth positive function and $g_*$ is the standard round metric on $\bS^n$. Assume that the scalar curvature of $\gamma$ satisfies $R(\gamma)\geq\tau$ for a fixed $\tau\in\R$. Then for any $s_0>a$ with $u'(s_0)>0$, there is a $\delta > 0$ and a metric $\widetilde{\gamma} = ds^2 + \widetilde u(s)^2 g_*$, such that $\widetilde \gamma = \gamma$ on $[s_0,\infty)\times \bS^n$ and $R(\widetilde{\gamma}\,)>\tau$ on $[s_0-\delta,s_0)\times \bS^n$. If in addition, $u(s_0) > c$ for some positive constant $c$ and $u''(s_0) > 0$, then $\widetilde u(s_0 - \delta) > c$ and $\widetilde u'(s_0 - \delta) < u'(s_0)$.
\end{lemma}
\begin{proof}
First, from \eqref{collar-general} and \eqref{scalar-collar-general} with $v \equiv 1$ and $E=u$, we know that the scalar curvature of a metric of the form
\begin{align*}
\gamma=ds^2+u(s)^2g_*
\end{align*}
is well-known (see, for example, \cite{CM}) to be given by
\begin{align}\label{eq-R-formula-simple}
R(\gamma)=\frac{n}{u^2}\( (n-1)-(n-1)(u')^2-2uu''\).
\end{align}

Notice that if $R(\gamma) > \tau$ already holds at $s=s_0$, nothing needs to be done and we can set $\widetilde\gamma = \gamma$ and choose $\delta \leq s_0 -a$ such that $R(\gamma) > \tau$ in $[s_0 - \delta, s_0]$. Otherwise, we seek to bend the metric near the cross-section $s=s_0$. We do this with a strictly increasing reparametrisation function $\s$ such that on some interval $[s_0-\delta,s_0)$, the metric
		\begin{align*}
		\widetilde{\gamma}\definedas ds^2+u(\s(s))^2g_*
		\end{align*}
		has scalar curvature $\widetilde{R}\definedas R(\widetilde{\gamma}\,) > \tau$, given via \eqref{eq-R-formula-simple} by	
	\begin{align*}
		\widetilde{R}=\frac{n}{u(\s)^{2}}\( (n-1)-(n-1)\(\frac{d}{d s}u(\s)\)^2-2u(\s)\frac{d^2}{ds^2}u(\s) \).
		\end{align*}	
Denoting $s$-derivatives by $\dot{}$ and $\s$-derivatives by $'$, expanding this gives
\begin{align*}
\widetilde{R}=\frac{n}{u(\s)^{2}}\( (n-1)-(n-1)\(u'(\s)\dot{\s}\)^2-2u(\s)\(u''(\s)\dot{\s}^2+u'(\s)\ddot{\s}  \) \) ,
\end{align*}
which can be rearranged to
\begin{align*}
\widetilde{R}&=\frac{n}{u(\s)^{2}}\( (n-1)-(n-1)\dot{\s}^2-2u(\s)u'(\s)\ddot{\s}\) \\ &\qquad +\frac{n \dot{\s}^2}{u(\s)^{2}}\((n-1)-(n-1)u'(\s)^2-2u(\s) u''(\s)\).
\end{align*}

Now note that the second group of terms is exactly the scalar curvature $R(\gamma)$ of the warped product metric $\gamma$ we started with, multiplied by $\dot{\s}^2$. As $R(\gamma)\geq\tau$ by assumption, we have
\begin{align*}
\widetilde{R}\geq u(\s)^{-2}\( n(n-1)-n(n-1)\dot{\s}^2-2nu(\s)u'(\s)\ddot{\s}\)+(\dot{\s}^2-1)\tau +\tau.
\end{align*}

That is, in order to achieve $\widetilde{R}=R(\widetilde{\gamma}\,)>\tau$ on $[s_{0}-\delta,s_{0})\times \bS^{n}$, we must choose $\s$ so that
\begin{align}\label{eq-sigmacond1}
		n(n-1)-n(n-1)\dot{\s}^2-2nu(\s)u'(\s)\ddot{\s}+u(\s)^2(\dot{\s}^2-1)\tau>0
		\end{align}
on $[s_{0}-\delta,s_{0})$. For this, it is sufficient to use the same function $\s$ as is used for bending the Schwarzschild manifold in the asymptotically flat case \cite{M-S,CCMM}. Define the translated bump function $\theta(s)\definedas1+\exp(-(s-s_0)^{-2})$, such that $\theta(s_0)=1$. For $\delta>0$ such that $a<s_{0}-\delta$ and $s\in[s_{0}-\delta,s_{0})$, define
\begin{align*}
\s(s)\definedas\int_{s_0-\delta}^{s}\theta(t)\,dt+C_\delta,
\end{align*}
where $C_\delta\definedas s_{0}-\int_{s_{0}-\delta}^{s_0}\theta(t)dt$. This choice of constant ensures that $\sigma$ can be smoothly extended to $\sigma(s)=s$ for $s\geq s_0$. With this choice of $\s$, \eqref{eq-sigmacond1} becomes		
\begin{align}\label{eq:s0}
(u(\s(s))^2\tau-n(n-1))\(2+e^{-\frac{1}{(s-s_0)^2}}\)+\frac{4nu(\s(s))u'(\s(s))}{(s_0-s)^3}>0
\end{align}
for all $s\in[s_{0}-\delta,s_{0})$.
		
We see that \eqref{eq:s0} is satisfied provided $\delta$ is taken sufficiently small, as the last term on the left-hand side is positive and dominates as $\delta$ goes to zero.
		
If in addition $u(s_0)> c > 0$, then $\widetilde{u}$ defined by $\widetilde{u}(s)\definedas u(\s(s))$ clearly satisfies $\widetilde u(s_0 - \delta) > c$, taking $\delta$ smaller if necessary. For the last assertion in the lemma, note that 
		\begin{align*}
		\widetilde u'' = \ddt{^2}{s^2} u(\s)  = u''(\s) \dot\s^2 + u'(\s)\ddot\s
		\end{align*}
		on $[s_0 - \delta, s_0]$, which can also be made positive by making $\delta$  smaller if $u''(\s(s_0)) > 0$. Then, in this case, $\widetilde u'$ is increasing on $[s_0-\delta,s_0]$, which implies $\widetilde u'(s_0 - \delta) < \widetilde u'(s_0) = u'(s_0) $.
	\end{proof}

	We now combine the above lemma with the concrete profile of AdS-Schwarz\-schild manifolds. As we do not intend to discuss the hyperbolic Hawking mass in higher dimensions, we will from now on restrict our attention to $n=2$.
	
	\begin{prop}\label{prop-gluing2}
		Let $f$ be a smooth positive function on an interval $[a,b]$ with $f'(b)>0$, $g_*$ be the standard metric on $\bS^2$, and define the metric $\gamma_f\definedas ds^2+f(s)^2g_*$ on $[a,b]\times \bS^2$. Suppose\\[-1.5ex]
		\begin{enumerate}[(i)]\itemsep1.5ex
			\item $\gamma_f$ has scalar curvature $R(\gamma_f)>-6$,
			\item $\S_b= \{b\}\times \bS^2$ has positive mean curvature, and
			\item $\mhh(\S_b)\geq -f(b)^3$.\\[-1.5ex]
		\end{enumerate}
		Then for any $m_e> \mhh(\S_b)$, there exists a rotationally symmetric, asymptotically hyperbolic $3$-manifold $(M,\gamma)$ with inner boundary $\partial M$ and scalar curvature $R(\gamma)\geq-6$, such that\\[-1.5ex]
		\begin{enumerate}[(I)]\itemsep1.5ex
			\item $(M,\gamma)$ is isometric to the spatial AdS-Schwarzschild manifold with mass $m_e$ outside of a neighbourhood of the inner boundary and
			\item there is a neighbourhood of the inner boundary that is isometric to $\left( [a,\frac{a+b}{2}]\times \bS^2,\gamma_f\right)$.\\[-1.5ex]
		\end{enumerate}
		In addition, if $f'>0$ on $[a,b]$ then $(M,\gamma)$ is foliated by mean convex CMC spheres.
	\end{prop}
	\begin{proof}
		First, recall there exists a coordinate chart in which the AdS-Schwarzschild metric of mass $m$ can be written as
		\begin{align*}
		{\gamma}_{m}\definedas ds^2+u_m(s)^2g_*,
		\end{align*}
		as in \eqref{sch-um} for $s\in[0,\infty)$, where $u_m>0$ satisfies\\[-1.5ex]
		\begin{itemize}\itemsep1.5ex
			\item $u_m(0)=r_+$ is the largest root of $r+r^3-2m$ if $m>0$,
			\item $u_m(0)=0$ if $m\leq0$,
			\item $u'_m=\sqrt{1+u_m^2-\frac{2m}{u_m}}>0$, on $\R^+$, and
			\item $u''_m=u_m+\frac{m}{u_m^2}$, on $\R^+$.\\[-1.5ex]
		\end{itemize}
\vspace{1ex}
By  \eqref{Hawking-level}, the hyperbolic Hawking mass of $\S_b=\{b\}\times\S$ is given by
		\begin{align}\label{eq-hawkingmass}
		m_{*}\definedas\mhh(\S_b)=\frac{f(b)}{2}\( 1+f(b)^2-f'(b)^2 \).
		\end{align}
		
		Now fix $m_e>\mhh(\S_b)$. We first consider the case when $m_e > 0$. In this case, $u_{m_e}'$ attains all values in $[0,\infty)$ and we can choose $s_0 > 0$ such that $u_{m_e}'(s_0) = f'(b) > 0$. From this, \eqref{eq-hawkingmass} and the expression for $u_{m_e}'$, we have		
		\begin{align*}
		m_e>\frac{f(b)}{2}\( 1+f(b)^2-(f'(b))^2 \)=\frac{f(b)}{2}(f(b)^2-u_{m_e}(s_0)^2)+\frac{f(b)}{u_{m_e}(s_0)}\,m_e,
		\end{align*}
		from which we have
		\begin{align*}
		& m_e\(1-\frac{f(b)}{u_{m_e}(s_0)}\)>\frac{f(b)}{2}\(f(b)^2-u_{m_e}(s_0)^2\)\\
		\Leftrightarrow\;& m_e(u_{m_e}(s_0)-f(b))>\frac{u_{m_e}(s_0)f(b)}{2}(f(b)-u_{m_e}(s_0))(f(b)+u_{m_e}(s_0))\\
		\Leftrightarrow\;& (u_{m_e}(s_0)-f(b))\( m_e+\frac{u_{m_e}(s_0)f(b)}{2}(f(b)+u_{m_e}(s_0)) \)>0,
		\end{align*}
and therefore $0<f(b)<u_{m_e}(s_0)$. In particular, we have established that the hypotheses (ii) and (iv) of Lemma~\ref{lemma-gluing} are met for $f_1=f$ on $[a,b]$ and $f_2=u_{m_e}$ on $[s_0,\infty)$.
		
		The remaining case to consider is when $m_e \leq 0$. To apply Lemma \ref{lemma-gluing} to $f_1$ and $f_2$ as given above, we require $u_{m_e}(s_0)>f(b)$. As the range of $u_{m_e}$ is all of $\mathbb R^+_0$, we are free to choose $s_{\veps_1}$ such that $u_{m_e}(s_{\veps_1})=f(b)+\veps_1$, for any $\veps_1>0$. We then have
		\begin{align*}
		u'_{m_e}(s_{\veps_1})^2=1+f(b)^2-\frac{2m_{e}}{f(b)}+O(\veps_1).
		\end{align*}
		Letting $\mu\definedas\frac{2m_{e}-2m_*}{f(b)}>0$, we have
		\begin{align*}
		u'_{m_e}(s_{\veps_1})^2&=1+f(b)^2-\frac{2m_*}{f(b)}-\mu+O(\veps_1)\\
		&=f'(b)^2-\mu+O(\veps_1),
		\end{align*}
and therefore for some $s_0=s_{\veps_1}$, choosing $\veps_1$ sufficiently small, we have $0 < u'_{m_e}(s_0)<f'(b)$ and $u_{m_e}(s_0) > f(b)$, recalling that $\mu>0$, which establishes hypotheses (ii) and (iv) of Lemma~\ref{lemma-gluing}. In order to show hypothesis (iii) of Lemma \ref{lemma-gluing}, we only need to show that $f'(b) < \sqrt{1 + 3f(b)^2}$. This trivially follows from \eqref{eq-hawkingmass} and the assumption that $m_* = \mhh(\S_b) > -f(b)^3$.
		
		To complete the proof, we simply observe that the AdS-Schwarzschild exterior can be bent slightly near $s=s_0$ by applying Lemma \ref{lemma-bending}, so that it has positive scalar curvature near $s=s_0$. As this can be done while preserving the other hypotheses of Lemma \ref{lemma-gluing}, we can directly apply Lemma \ref{lemma-bending} to complete the proof.\\ 
		
		If $f'(s)>0$ everywhere on $[a,b]$ then it is clear from \eqref{eq-Hs} that each $\Sigma_s$ has constant positive mean curvature.
	\end{proof}
	
	In order to prove Theorem \ref{thm-main}, we next construct appropriate collars so that we may apply Proposition \ref{prop-gluing2} to construct admissible extensions of given Bartnik data. More precisely, Proposition \ref{prop-gluing2} then asserts that $\gamma_f$ can be smoothly glued to an AdS-Schwarzschild manifold of any mass $m_e > \mhh(\S_b)$, preserving the lower bound on the scalar curvature.

	\section{Extensions with minimal boundary}\label{SHzero}
	In \cite{M-S}, Mantoulidis and Schoen constructed asymptotically flat extensions of Bartnik data with $H\equiv0$, with non-negative scalar curvature. Their construction works for any metric $g$ on $\S$ such that the first eigenvalue of $-\Lap_g + K(g)$ is strictly positive, where $K(g)$ denotes the Gaussian curvature of $(\S,g)$. This condition arises naturally in the theory of minimal surfaces in the following way. A compact minimal surface $\S$ in a Riemannian manifold $(M,\gamma)$ is said to be \emph{strictly stable} if
	\begin{align}\label{stab-eq}
	\int_{\S} |\grad \varphi|^2 > \int_{\S} ( |A|^2 + \Ric(\nu,\nu)) \varphi^2
	\end{align}
	for all smooth $\varphi\not\equiv0$ on $\S$. Here $\nu$ and $A$ denote the outward unit normal and the second fundamental form of $\Sigma$ in $M$, respectively, and $\Ric$ is the Ricci tensor of $(M,\gamma)$. Using the Gauss equation, one can rewrite the right-hand side of \eqref{stab-eq} and then the strict stability condition implies
	\begin{align*}
	\int_{\S} -\varphi \Lap_g \varphi  + K(g) \varphi^2 > \int_{\S}  \dfrac{R(\gamma) + |A|^2}{2} \varphi^2  \geq  \int_{\S} \dfrac{R(\gamma)}{2} \varphi^2
	\end{align*}	
	for all smooth $\varphi\not\equiv0$ on $\S$. So if $(M,\gamma)$ has non-negative scalar curvature, the condition of strict stability implies that the first eigenvalue of $-\Lap_g + K(g)$ is strictly positive.
	
	In our setting, $(M,\gamma)$ is an asymptotically hyperbolic $3$-manifold with $R(\gamma) \geq -6$. Strict stability therefore implies that the first eigenvalue of $-\Lap_g + K(g)$ is strictly greater than $-3$.
	
	It would then be natural to consider the set of metrics described as
	\begin{align*}
	\M= \{ g \,\, \textmd{metric on $\S$} \, : \, \lambda_1(-\Lap_g + K(g)) > -3 \},
	\end{align*}
	where $\lambda_1 = \lambda_1 (-\Lap_{g} + K(g))$ denotes the first eigenvalue of $-\Lap_{g} + K(g)$.  However, the path of metrics used by Mantoulidis and Schoen does not seem to be compatible with this class. Furthermore, the authors are not aware of any geometric flow known to preserve a (negative) lower bound for the smallest eigenvalue of the operator $-\Lap_{g} + K(g)$. Yet, if $\M$ is path-connected, the proofs of Theorem \ref{thm-main} and Corollary \ref{coro1} would work with the weaker hypothesis  $\lambda_1(-\Delta_{g_o}+K(g_o))>-3$. To see this, take any path in $\M$ connecting the given metric $g_{o}$ to the round metric of area $4\pi$ and repeat the procedure used in \cite{M-S}.

Therefore, instead of considering $\M$, we consider the same set $\M^+$ of metrics considered in \cite{M-S}. That is, assume that $g_o \in \M^+$, where
\begin{align*}
	\M^+= \{ g \,\, \textmd{metric on $\S$} \, : \, \lambda_1(-\Lap_g + K(g)) > 0 \}.
	\end{align*}
This set is path-connected as can be seen using the path $g(t)$ given by the Uniformisation Theorem \cite[Proposition 1]{M-S} which connects a given metric $g(0)=g_{o}$ to the round metric $g(1)=g_{*}$ of area $4\pi$. In addition, this path can be modified using \cite[Lemma 1.2]{M-S} so that 
\begin{enumerate}\itemsep1.5ex
\item $g'(t) = 0$ on $[\theta,1]$, for some $0<\theta<1$,
\item $\tr_{g(t)}{g}'(t)=0$, and
\item $g(1)$ is round.
\end{enumerate}
In particular, $g(t)$ has the same area as $g(0)=g_{o}$ for all $t$, as (2) implies that the area form is preserved. 

\vspace{2ex}

As in \cite{M-X} (see also \cite{CCMM}), we define the following scale-invariant constants $\alpha$ and $\beta$ associated with a path of metrics $\{g(t)\}_{t \in [0,1]}$:
\begin{align}
\begin{split}\label{alpha-beta-def}
r_o &\definedas\sqrt{\dfrac{| \S|_{g_o}}{4\pi}},\\
\a &\definedas  \dfrac{1}{4} \max_{[0,1]\times\S} |g'(\cdot)|^2_{g(\cdot)},\\
\beta &\definedas \min_{[0,1]\times\S} r_o^2\, K(g(\cdot)).
\end{split}
\end{align}

Note that the quantities $\a$ and $\b$ depend on the chosen path $\{g(t)\}_{t \in [0,1]}$, and give a measure of how far the metric $g_o$ is from being round. Specifically, Miao and Xie show that if $g_o$ is sufficiently close to $g_*$ in the $C^{2,\tau}$ topology then we can fix the path such that $\alpha<C\|g_o-g_*\|_{C^{0,\tau}}$ (see Proposition 4.1 of \cite{M-X}). Note that in the case considered by Miao and Xie, it is assumed that $K(g_o)>0$. However, this is guaranteed here by the $C^{2,\tau}$-closeness of $g_o$ to $g_*$.

Note that by condition (1) above, $g(t) = r_o^2 g_*$ for all $\theta \leq  t \leq 1$. While $\alpha$ appears in the proof of Theorem \ref{thm-main}, $\beta$ will only come into play in Section \ref{SHnonzero}. 

Working with this modified path, we now proceed to prove the following theorem.
\begin{thm}\label{thm-main}
Let $(\S \cong \bS^2,g_o,H_o=0)$ be  Bartnik data satisfying $\lambda_1(-\Delta_{g_o}+K(g_o))>0$, where $\lambda_1(-\Delta_{g_o}+K(g_o))$ denotes the first eigenvalue of the operator $-\Delta_{g_o}+K(g_o)$ on $\S$ and $K(g_o)$ denotes the Gaussian curvature of $g_o$. Then for any 
\begin{align*}
m > \mhh(\S,g_o,H_o=0)=\dfrac{1}{2}\( \(\dfrac{|\S|_{g_o}}{4 \pi}\)^{1/2} + \(\dfrac{|\S|_{g_o}}{4\pi}\)^{3/2} \), 
\end{align*}
there is an asymptotically hyperbolic Riemannian manifold $(M,\gamma)$ with $R(\gamma) \geq -6$  such that\\[-2.125ex]
\begin{enumerate}[(i)]\itemsep1.5ex
\item the boundary $\partial M$ is minimal and isometric to $(\S,g_o)$,
\item outside a compact set, $M$ coincides with the spatial AdS-Schwarzschild manifold of mass~$m$, and
\item $M$ is foliated by mean convex spheres that eventually coincide with the coordinate spheres in the spatial AdS-Schwarzschild manifold.\\[-1.5ex]
\end{enumerate}
\end{thm}

\begin{remark}\label{rem-K-3}
Instead of assuming $\lambda_1(-\Delta_{g_o}+K(g_o))>0$, we could assume $K(g_o)>-3$ and obtain the same conclusion; this can be seen by following the proof of Theorem \ref{thm-main} using Ricci Flow in two dimensions, as described in Lemma \ref{Ricci2D} below, instead of the Uniformisation Theorem.
\end{remark}

\begin{proof}
Let $m > \mhh(\S,g_o,H_o=0)$, then there exists $\delta > 0$ such that 
\begin{align*}
\mhh(\S,g_o,0) < \mhh(\S,g_o,0) + \delta < m.
\end{align*}
Let $\{g(t)\}_{t\in[0,1]}$ be the path described above satisfying (1) and (2). In view of Proposition~\ref{prop-gluing2}, we want to construct an appropriate collar extension of $(\S,g_o,H_o\equiv0)$ that is rotationally symmetric close to the outer boundary, which explains the motivation for condition (1).

Recall that we suppress the dependence of $x\in\S$ of various quantities throughout. For each $t \in [0,1]$, let $u(t,\cdot) > 0$ on $\S$ denote a positive eigenfunction corresponding to the first eigenvalue $\lambda(t) = \lambda_1(t)$ of $-\Lap_{g(t)} + K(g(t))$; such choice of eigenfunctions can be made so that $u$ varies smoothly on $[0,1] \times \S$ and $u(t,\cdot)$ is normalised to have unit $L^2$ norm with respect to the area element $dV_{g(t)}$ (see \cite[Lemma A.1]{M-S}). Now let $0 < \veps <1$ and consider a metric of the form \eqref{collar-general}, with $E(t)\definedas(1+\veps t^2)^{1/2}$ and
 $v(t,\cdot)^2\definedas A^2u(t,\cdot)^2$, for some constant $A > 0$ to be determined.
  Then
\begin{align}\label{E-derivatives}
\begin{split}
E'(t) &= \dfrac{\veps t}{(1+\veps t^2)^{1/2}},\\
E''(t) &= \dfrac{\veps}{(1+\veps t^2)^{3/2}},
\end{split}
\end{align}
and the collar extension is given on $[0,1] \times \bS^2$ by
\begin{align}\label{collar-minimal}
\gamma_c =  A^2u(t,\cdot)^2 dt^2 + (1+\veps t^2)g(t).
\end{align}

Using \eqref{scalar-collar-general}, the scalar curvature of $\gamma_c$ as described in \eqref{collar-minimal} can be estimated as
\begin{align*}
&R( \gamma_c) + 6 \\
&\quad\quad= 2v(t,\cdot)^{ -1} \( - \Lap_{E(t)^2g(t)} v(t,\cdot) + K(E(t)^{2}g(t)) v(t,\cdot)  \)\\
&\quad\quad\quad + v(t,\cdot)^{ -2}\left[  \dfrac{-2 E'(t)^2 -4  E(t)E''(t)}{E(t)^2} - \dfrac{1}{4} |g'(t)|^2_{g(t) }  + 4 \dfrac{\pr_t v(t,\cdot)}{v(t,\cdot)} \dfrac{E'(t)}{E(t)}\right] +6 \\
&\quad\quad=2u(t,\cdot)^{-2}A^{ -2}E(t)^{-2}\\
&\quad\quad\quad\times\left[  A^2u(t,\cdot)^2( \lambda(t) + 3E(t)^2) -\veps - \dfrac{ \veps }{E(t)^2} - \dfrac{1}{8} |g'(t)|^2_{g(t) }E(t)^2 + 2\veps t \dfrac{\pr_t u(t,\cdot)}{u(t,\cdot)} \right] \\
&\quad\quad> 2u(t,\cdot)^{-2}A^{ -2}E(t)^{-2}\left[  A^2\inf\limits_{[0,1] \times \S} u^2( \lambda + 3) -2 - \a - 2 \sup\limits_{[0,1] \times \S}\left|\dfrac{\pr_t u}{u}\right| \right], 
\end{align*}
where we used $-\Lap_{E(t)^2g(t)}v(t,\cdot) + K_{E(t)^2g(t)}v(t,\cdot)=E(t)^{-2} \lambda(t) v(t,\cdot)$. 

Since $u >0$ and because $\lambda(t)> 0$ for all $t\in[0,1]$, it follows that $\inf\limits_{[0,1] \times \S} u^2\,( \lambda + 3) > 0$. Now pick $A>0$ such that
\begin{align} \label{lambda3}
A^2\inf\limits_{[0,1] \times \S} u^2\,(\lambda + 3) -2 - \a - 2 \sup\limits_{[0,1] \times \S}\left|\dfrac{\pr_t u}{u}\right|  > 0
\end{align}
so that $R(\gamma) > -6$.

Since $g(t)=r_o^2 g_*$ for all $t \in [\theta,1]$, we know that $u\equiv u(1)$ is a fixed positive constant on $\S_t = \{ t \} \times \S$ for all $t \in [\theta,1]$. Using \eqref{Hawking-level}, its Hawking mass is
\begin{align*}
\mhh(\S_t ) &=  \dfrac{E(t) r_o}{2}\(1-  \dfrac{r_o^2 E'(t)^2}{v(t,\cdot)^2}  + r_o^2 E(t)^2 \) \\
&=\dfrac{(1+\veps t^2)^{1/2} r_o}{2}\(1- \dfrac{r_o^2 \veps^2 t^2}{A^2 u(1)^2(1+\veps t^2)} + r_o^2 (1+\veps t^2)    \),
\end{align*}
and in particular, for any $0 < \veps <1$ such that $\dfrac{\veps}{1+\veps} < A^2 u(1)^2$, 
\begin{align}
\begin{split}\label{eq-Hawking-positiveeq}
\mhh(\S_1 ) &= \dfrac{(1+\veps )^{1/2} r_o}{2}\(1- \dfrac{r_o^2 \veps^2 }{A^2 u(1)^2(1+\veps)} + r_o^2 (1+\veps )   \) \\
&\leq \dfrac{(1+\sqrt{\veps}) r_o}{2}\(1- \dfrac{r_o^2 \veps^2 }{A^2 u(1)^2(1+\veps)} + r_o^2 (1+\veps )    \) \\
&\leq \mhh(\S_0 )  + \sqrt{\veps}\, C,
\end{split}
\end{align}
where $C$ is a positive constant depending only on $r_o$, $A$, and $u(1)$. The mean curvature of $\S_t$ is given via \eqref{eq-Hs} by
\begin{align}\label{meanc-minimal}
H(t)=\dfrac{2E'(t)}{v(t,\cdot) E(t)}=\dfrac{2\veps t}{A u(t,\cdot) E(t)^2},
\end{align}
so that $H(0)=0$ and $H(t)>0$ for $t \in (0,1]$. Notice that by \eqref{eq-Hawking-positiveeq} and our choice of $\veps$, we have $\mhh(\S_1 )>0$.
In order to apply Proposition \ref{prop-gluing2}, we need to perform a change of variable to bring $\gamma_c$ to the form $\gamma_c = ds^2 + f(s)^2 g_*$. After the change of variable $s(t) \definedas Au(1)t$, we can write $\gamma_c$ on $[Au(1) \theta, A u(1)] \times \bS^2$ as the rotationally symmetric metric
\begin{align*}
\gamma_c  =  ds^2 + \( 1 + \dfrac{\veps}{A^2 u(1)^2} s^2 \)r_o^2 g_*.
\end{align*}
Therefore, since $m > \mhh(\S,g_o,H_o\equiv0) + \delta \geq \mhh(\S_1)$, we can apply Proposition \ref{prop-gluing2} to obtain an asymptotically hyperbolic extension $(M,\gamma)$ of mass~$m$ (since it coincides with an AdS-Schwarzschild manifold outside a compact set), with $R(\gamma) > -6$, and such that it is foliated by mean convex spheres and $\partial M$ is isometric to $(\S,g_o)$ and it is minimal by \eqref{meanc-minimal}.
\end{proof}

As pointed out at the end of Subsection \ref{sec:computations}, each asymptotically hyperbolic extension constructed in Theorem \ref{thm-main} gives an upper bound for $ \m_{B}^{AH}(\S,g_o,H_o\equiv0)$ and therefore the following corollary holds.

\begin{coro} \label{coro1}
Let $(\S \cong \bS^2,g_o,H_o=0)$ be  Bartnik data satisfying $\lambda_1(-\Delta_{g_o}+K(g_o))>0$, where $\lambda_1(-\Delta_{g_o}+K(g_o))$ is defined as in Theorem \ref{thm-main}. Then the hyperbolic Bartnik mass of $(\S,g_o,H_o=0)$ satisfies
\begin{align*}
\m_B^{AH}(\S,g_o,H_o=0) \leq \mhh(\S,g_o,H_o=0)=\dfrac{1}{2}\( \(\dfrac{|\S|_{g_o}}{4 \pi}\)^{1/2} + \(\dfrac{|\S|_{g_o}}{4\pi}\)^{3/2} \).
\end{align*}
\end{coro}

\begin{remark}
As explained in Remark \ref{rem-K-3}, Corollary \ref{coro1} remains true if the assumption $\lambda_1(-\Delta_{g_o}+K(g_o))>0$, is replaced by the assumption $K(g_o)>-3$.
\end{remark}

\section{Collars with $K(g_o) > -3$}\label{SHnonzero}
In \cite{M-X} (see also \cite{CCMM}), a collar is constructed for CMC Bartnik data, modelled on the spatial Schwarzschild metric in order to obtain good control on the Hawking mass along the collar. Here, instead of controlling the usual Hawking mass along the collar, we would like to control the hyperbolic version of the Hawking mass \eqref{eq-AHHdef}. For this, the appropriate manifold to model our collars on is the spatial AdS-Schwarzschild manifold, which we recall from Section \ref{SS-adsschwarzschild} can be expressed as
\begin{align*}
g_{m,b}=ds^2+\umb(s)^2g_*.
\end{align*}
Note that if $m>0$ then we must choose the parameters $m$ and $b$ such that $r_o>r_+$.

In the asymptotically flat case, the CMC Bartnik data is assumed to have positive Gaussian curvature; a condition motivated by the stability condition as discussed in Section \ref{SHzero}. In the asymptotically hyperbolic case, where we consider manifolds with scalar curvature bounded below by $-6$, the analogous lower bound for the Gaussian curvature of CMC surfaces is $-3$. That is, we consider metrics $g_o$ on $\S$ satisfying $K(g_o) > -3$ throughout this section.

Since the condition $K(g_o)>-3$ is not preserved under rescaling, we adopt a different method to construct the smooth path of metrics $\{g(t)\}_{t\in[0,1]}$ used in the collar construction. We will use Ricci flow and exploit its special features in two dimensions.

\subsection{Smooth paths of metrics with a lower bound on the Gaussian curvature}
We prove the following lemma.
\begin{lemma} \label{Ricci2D}
Given a metric $g_o$ on $\S \cong \bS^2 $ with $K(g_o) > -\k$, where $\k>0$, there exists a smooth path of metrics $\{g(t)\}_{t\in[0,1]}$ such that\\[-1.5ex]
\begin{enumerate}\itemsep1.5ex
\item $g(0)=g_o$ and $g(t)$ is round on $[\theta,1]$ for some fixed $0<\theta<1$,
\item $\tr_{g(t)} g'(t) = 0$ for all $t \in [0,1]$, and
\item $K(g(t)) > -\k$ for all $t \in [0,1]$.
\end{enumerate}
\end{lemma}

\begin{proof}
It is well-known that in dimension 2, the area-preserving (or ``normalised") Ricci flow is equivalent to the flow
\begin{align*}
\begin{split}
\dd{}{t} g(t) &= 2(-K(g(t)) +  r_o^{-1})g(t),  \\
g(0)&=g_o,
\end{split}
\end{align*}
where $r_o$ is given as before by $| \S |_{g_o}=4\pi r_o^2$. Using the results of Hamilton \cite{Hamilton2} and Chow \cite{Chow91}, the solution $g(t)$ exists for all time and converges exponentially to a round sphere as $t \to \infty$.

We now turn our attention to the Gaussian curvature along the flow. The evolution of the Gaussian curvature along the normalised Ricci flow is given by the parabolic equation (see \cite[Eq. 3.1]{Hamilton2})
\begin{align*}
\dd{}{t} K(g(t)) = \Lap_{g(t)} K(g(t)) + 2K(g(t))(K(g(t))-r_o^{-1}).
\end{align*}

Using a maximum principle (combined with a comparison to an ODE) as in \cite{C-K} (see also~\cite{S-HT} for some more details), one has the following estimate for all $t\geq0$
\begin{align*}
K(g(t)) \geq \min\lbrace0,\min\limits_{\S} K(g(0))\rbrace.
\end{align*}
Since our initial metric satisfies $K(g_o) > -\k$, by the normalised Ricci flow we obtain a path $\{g(t)\}_{t\geq0}$ such that $g(0)=g_o$, $\lim_{t\to\infty}g(t)$ is round, $| \S |_{g(t)}$ is constant, and the Gaussian curvature satisfies $K(g(t)) > -\k$ for all $t\geq0$. The exponential convergence of the flow allows a reparametrisation to obtain a path $\{g(t)\}_{t\in[0,1]}$ such that $g(0)=g_o$ and $g(1)$ is a round sphere of area $4\pi r_o^2$. Notice that after a further reparametrisation we can ensure that (1) is satisfied while (3) still holds. To alter the flow so that (2) is satisfied without violating (1) and (3), one can solve a family of partial differential equations on $\S$ to obtain a $1$-parameter family of diffeomorphisms $\{ \phi_t \}$ on $\S$ such that the path defined by $\{\phi_t^*(g(t))\}$ satisfies (2). We refer the interested reader to \cite{M-S} for details.
\end{proof}

\subsection{Collar extensions}
The collar manifold that we consider is again $[0,1]\times\S$ equipped with a metric of the form \eqref{collar-general}; specifically, we use the metric
\begin{align}\label{eqCollar}
\gamma=A^2dt^2+r_o^{-2}\umb (Akt)^2g(t),
\end{align}
with constants $A,k>0$. By \eqref{eq-Hs}, we compute the mean curvature of $\S_t =\{ t \} \times \S$ to~be
\begin{align*}
H(t)\definedas H(\S_t)=\frac{2k}{\umb(Akt)}\sqrt{1+b\umb (Akt)^{2}-\frac{2m}{\umb (Akt)}}.
\end{align*}
In order to induce the given the Bartnik data at $t=0$, we need to prescribe the mean curvature of the boundary $\S_0$ to be the given constant $H_o$. It is elementary to see that in order to achieve this, we must fix $k$ in terms of $m$ and $b$ as
\begin{align}\label{eq-kmHpos}
k^2&=\frac{H_o^2r_o^2}{4}\( 1+br_o^2-\frac{2m}{r_o} \)^{-1}, 
\end{align}
which is equivalent to
\begin{align} \label{eq-k-H0-m}
k^2\(1-\frac{2m}{r_o}\)=\frac{(H_o^2-4k^2b)r_o^2}{4}.
\end{align}

As earlier, we compute the scalar curvature using \eqref{scalar-collar-general} and obtain
\begin{align*}
&R(\gamma)  + 6  \\
&\quad= 2u_{m,b}(Akt)^{-2}\left[r_o^2K(g(t))\right] \\
&\quad\quad +A^{-2}\left[-2\dfrac{A^2k^2 u_{m,b}'(Akt)^2 r_o^{-2} +2  A^2 k^2 u_{m,b}(Akt) r_o^{ -2} u_{m,b}''(Akt)}{u_{m,b}(Akt)^2 r_o^{-2}}   -\dfrac{1}{4} |\dot{g}(t)|^2_{g(t)} \right] +6 \\
&\quad= 2u_{m,b}(Akt)^{-2}\left[r_o^2K(g(t)) + 3u_{m,b}(Akt)^2\right]
+A^{-2}\left[-2A^2k^2\dfrac{1 + 3 bu_{m,b}(Akt)^2  }{u_{m,b}(Akt)^2 }   -\dfrac{1}{4} |\dot{g}(t)|^2_{g(t)} \right]  \\
&\quad = 2u_{m,b}(Akt)^{-2}\left[ r_o^2K(g(t))+ 3u_{m,b}(Akt)^2  -k^2 -3k^2b u_{m,b}(Akt)^2 - \dfrac{u_{m,b}(Akt)^2}{8 A^2} |\dot{g}(t)|^2_{g(t)}\right ].
\end{align*}
Using $\alpha$ and $\beta$ as defined by \eqref{alpha-beta-def}, we find
\begin{align}
\begin{split}\label{estimateScal-mb}
R(\gamma)  + 6 &\geq  2u_{m,b}(Akt)^{-2}\left[ (\b + 3u_{m,b}(Akt)^2)  -k^2 -3k^2 b u_{m,b}(Akt)^2 - \dfrac{\a}{2 A^2} u_{m,b}(Akt)^2 \right ]\\
 &\geq  2u_{m,b}(Akt)^{-2}\left[ (\b + 3r_o^2)  -k^2 -3k^2 b u_{m,b}(Akt)^2 - \dfrac{\a}{2 A^2} u_{m,b}(Akt)^2 \right ] .
\end{split}
\end{align}

Now, provided that $\b + 3r_o^2 >0$, we have some freedom in choosing parameters $A>0$, $m \in \R$, and $b \geq 0$ to control the remaining terms and ensure $R(\gamma)\geq-6$. We also would like to control the hyperbolic Hawking mass along the cross-sections $\S_t$ of our collars, which is given by (see \eqref{Hawking-level})
\begin{align}
\begin{split}\label{eq-HawkingHpos}
\mhh(\S_t) 
&=\dfrac{u_{m,b}(Akt)}{2}\(1- k^2 (u_{m,b}'(Akt))^2 +  u_{m,b}(Akt)^2\) \\
&=\dfrac{u_{m,b}(Akt) }{2}\(1- k^2 \(1+ bu_{m,b}(Akt)^2 - \dfrac{2m}{u_{m,b}(Akt)}\) +  u_{m,b}(Akt)^2\)\\
&=\dfrac{u_{m,b}(Akt) }{2}\(1- k^2 + u_{m,b}(Akt)^2(1 - k^2 b)\)  + k^2 m. 
\end{split}
\end{align}
\vspace{1ex}

We next turn to consider different choices of the parameters $m$, $b$, and $A$ in order to get good estimates on the hyperbolic Hawking mass at the end of the collar. It will be useful to choose $m<0$. We will first consider the case where $b=0$, which corresponds to using the profile curves of Schwarzschild manifolds of negative mass $m$. We will use \eqref{eq-kmHpos} to express $k$ as a function of $m$ and will then take $\vert m\vert$ very large. This approach shows the existence of admissible extensions with the desirable properties stated in Theorem~\ref{thm-extension1}, given below. Such collars therefore give estimates on the Bartnik mass with these desirable properties, as stated in Corollary \ref{coro-main1}.

After stating Theorem \ref{thm-extension1} and Corollary \ref{coro-main1}, we consider the case $b>0$ corresponding to using the profile curves of AdS-Schwarzschild manifolds of negative mass $m$ and cosmological constant $\Lambda=-3b$. Geometrically, $\Lambda=-3b$ means that the AdS-Schwarzschild manifolds are asymptotic to a hyperboloid of ``radius'' $1/\sqrt{b}$. We will then pick a (small) number $\delta>0$, impose that $k^{2}b=\delta$, express $b$ and $k$ in terms of $m$ (and $\delta$) via \eqref{eq-kmHpos}, and proceed to take $\vert m\vert$ very large. This will then lead to existence of admissible extensions with the desirable properties stated in Theorem \ref{thm-extension2}, and therefore to the estimate on the Bartnik mass stated in Corollary \ref{coro-main2}. 

Note that there is nothing inherently special about the choices of parameters we use. One would obtain similar estimates with different choices of $m$ and $b$, however we are not aware of any choices of parameters that give qualitatively different estimates to those obtained here.\\[-2ex]

\paragraph*{\emph{\underline{Case $b=0$}}}
Recall that we choose $m<0$. As described above, we will be interested in considering collars where $|m|$ is very large and $k$ is given in terms of $m$ via \eqref{eq-kmHpos}. In this case, we can estimate the area radius at the end of our collar as follows. We have
\begin{align*}
u_{m,0}'(s) = \sqrt{ 1- \dfrac{2m}{u_{m,0}(s)}} < \sqrt{ 1- \dfrac{2m}{r_o}},
\end{align*}
for $s>0$ which implies
\begin{align}\label{um0Ak-est}
u_{m,0}(Ak) <  \sqrt{ 1- \dfrac{2m}{r_o}}Ak + r_o,
\end{align}
and from \eqref{eq-k-H0-m} we write this as
\begin{align}\label{eq-ubound-m0}
u_{m,0}(Ak)< \(\frac{H_oA}{2}+1\)r_o.
\end{align}

Using \eqref{um0Ak-est}, we next estimate the scalar curvature of the collar, so that we may choose the remaining parameters so that we have the scalar curvature strictly bounded below by $-6$:
\begin{align*}
&R(\gamma)  + 6 \\
&\quad\geq  2u_{m,0}(Akt)^{-2}\left[ \b + 3r_o^2  -k^2 - \dfrac{\a}{2 A^2} \,u_{m,0}(Akt)^2 \right ]  \\
&\quad>  2u_{m,0}(Akt)^{-2}\left[ \b + 3r_o^2  -k^2 - \a \( \( 1 -\dfrac{2m}{r_o} \)  k^2 + \dfrac{r_o^2}{A^2}  \)\right ]  \\
&\quad \geq  2u_{m,0}(Akt)^{-2}\left[ \b + 3r_o^2  -k^2\(1+\a\( 1 -\dfrac{2m}{r_o} \) \) - \a \dfrac{r_o^2}{A^2}\right ]. 
\end{align*}

It follows that we must choose $A>A_o$, where $A_o$ is defined in terms of $m$ by
\begin{align}\label{eq-A0}
A_o \definedas r_o\( \dfrac{\a}{ \b + 3r_o^2  -k^2\(1+\a\( 1 -\frac{2m}{r_o} \) \)}  \)^{1/2},
\end{align}
for which to be defined, we must impose
\begin{align*}
\b + 3r_o^2  &> k^2\(1+\a\( 1 -\dfrac{2m}{r_o} \) \) \\
&=\dfrac{H_o^2 r_o^2}{4}\( 1 -\dfrac{2m}{r_o} \)^{-1}  +\a  \dfrac{H_o^2 r_o^2}{4}. 
\end{align*}
Provided that the initial data satisfies
\begin{align*}
\dfrac{H_o^2 r_o^2}{4} < \dfrac{\b + 3r_o^2 }{\a}
\end{align*}
(or $\a=0$), this can always be ensured, simply by choosing $|m|$ sufficiently large.

Now, given such a collar, we must estimate the hyperbolic Hawking mass at the end of the collar where we will later glue on an AdS-Schwarzschild exterior. We estimate the hyperbolic Hawking mass of $\Sigma_1 = \{ 1 \} \times \Sigma$ by
\begin{align*}
\mhh(\Sigma_1) &= \dfrac{u_{m,0}(Ak)}{2} \(1 -k^2 + u_{m,0}(Ak)^2 \) +k^2m \\
&\leq  \( \dfrac{H_oA}{2}+ 1 \) \dfrac{r_o}{2} \( 1 -k^2 +  \( \dfrac{H_oA}{2}+ 1 \)^2 r_o^2   \) +k^2m \\
&=\( \dfrac{H_oA}{2}+ 1 \) \mhh(\S_o) - \dfrac{H_oA}{2} k^2m  + \dfrac{r_o^3}{2}\( \dfrac{H_oA}{2}+ 1 \)\( \( \dfrac{H_oA}{2}+ 1 \)^2   - 1 \),
\end{align*}
where the inequality follows from \eqref{eq-ubound-m0}. Observe that this inequality holds for all $m < 0$, $k=k(m)$ given by \eqref{eq-kmHpos}, and $A > A_o=A_o(m)$ given by \eqref{eq-A0}.

As a consequence, it also holds for $A=A_o$, so that
\begin{align*}
\mhh(\S_1) \leq \( \dfrac{H_oA_o}{2}+ 1 \) \mhh(\S_o) - \dfrac{H_oA_o}{2} k^2m  + \dfrac{r_o^3}{2}\( \dfrac{H_oA_o}{2}+ 1 \)\( \( \dfrac{H_oA_o}{2}+ 1 \)^2   - 1 \)
\end{align*}
holds for each fixed $m<0$.

Now let $\veps>0$. In order to estimate the hyperbolic Hawking mass with \eqref{eq-HawkingHpos} using \eqref{eq-kmHpos}, we first compute
\begin{align*}
	\lim_{m\to -\infty}k^2m&=-\frac{r_o}{2}\,\frac{H_o^2r_o^2}{4},\\
	A_{-\infty} \definedas \lim_{m\to -\infty} A_0&= r_o \(  \dfrac{\alpha}{\beta + 3r_o^2 - \alpha \frac{H_o^2 r_o^2}{4}} \)^{1/2}.
\end{align*}
Thus, taking $|m|$ sufficiently large, this leads to
\begin{align}
\begin{split}\label{estimate1}
\mhh(\S_1) &\leq \( \dfrac{H_o A_{-\infty}}{2}+ 1 \) \mhh(\S_o)  \\
&\quad + \dfrac{H_or_o^3 A_{-\infty}}{4}\left[ \dfrac{H_o^2 }{4}  +\( \dfrac{H_o A_{-\infty}}{2}+ 1 \)\(  \dfrac{H_o A_{-\infty}}{2} + 2 \)\right] + \veps.
\end{split}
\end{align}
\vspace{1ex}

We have thus constructed a family of collar extensions $([0,1]\times\S,\gamma)$ parametrised by $m<0$ which have scalar curvature $R(\gamma)>-6$ and hyperbolic Hawking mass estimated from above as in \eqref{estimate1} for any $\veps>0$ whenever $\vert m\vert$ is suitably large depending on $\veps$.

In order to glue these collars to an AdS-Schwarzschild manifold using Proposition \ref{prop-gluing2}, we proceed as in Section \ref{SHzero}. The only difference is that here, we need to ensure that
\begin{align*}
\mhh(\Sigma_{1}) \geq  -u_{m,0}(Ak)^3.
\end{align*}
This is equivalent to the condition
\begin{align*}
1 + 3u_{m,0}(Ak)^2   \geq k^2\(1 -\dfrac{2 m}{u_{m,0}(Ak)}\),
\end{align*}
which by \eqref{eq-kmHpos} and $u_{m,0}(Ak) \geq r_o$ can be ensured by restricting $H_o$ to satisfy
\begin{align} 
\frac{H_o^2r_o^2}{4}\leq 1+3r_o^2.
\end{align}
This simply says that the initial hyperbolic Hawking mass cannot be too negative.

That is, Proposition \ref{prop-gluing2} gives us an admissible asymptotically hyperbolic extension from given Bartnik data with controlled mass. In particular, we have established the following theorem.
\begin{thm}\label{thm-extension1}
Let $(\S\cong\bS^2,g_o,H_o)$ be Bartnik data such that the Gaussian curvature of $g_o$ satisfies $K(g_o) > -3$ and $H_o$ is a positive constant. If $H_o$ and the constants $r_o$, $\alpha$, and $\beta$ defined by \eqref{alpha-beta-def} satisfy
	\begin{align}
		\frac{H_o^2r_o^2}{4}<\min\( 1+3r_o^2,\frac{\beta+3r_o^2}{\alpha} \),
	\end{align}
then for any 
\begin{align*}
m > m_*,
\end{align*}
where $m_*$ is defined by
\begin{align*} 
m_*\definedas \( \xi + 1 \) \mhh(\S,g_o,H_o)  
+ \dfrac{r_o^3}{2}\xi\left[ \dfrac{H_o^2 }{4}  +\( \xi + 1 \)\(  \xi+ 2 \)\right],
\end{align*}
with
	\begin{align*}
	\xi\definedas \dfrac{H_or_o}{2}\(\dfrac{\a}{(\b + 3r_o^2) - \a \frac{H_o^2 r_o^2}{4}} \)^{1/2},
	\end{align*}
there is an asymptotically hyperbolic Riemannian manifold $(M,\gamma)$ with $R(\gamma) \geq -6$  such that\\[-1.5ex]
\begin{enumerate}[(i)]\itemsep1.5ex
\item the boundary $\partial M$ is isometric to $(\S,g_o)$ and has constant mean curvature $H_o$,
\item outside a compact set, $M$ coincides with the spatial AdS-Schwarzschild manifold of mass $m$, and
\item $M$ is foliated by mean convex spheres that eventually coincide with the coordinate spheres in the spatial AdS-Schwarzschild manifold.\\[-1.5ex]
\end{enumerate}
\end{thm}

It is clear that the Riemannian 3-manifolds $(M,\gamma)$ obtained in Theorem \ref{thm-extension1} are admissible extensions of the Bartnik data $(\S\cong\bS^2,g_o,H_o)$, and thus they provide upper bounds for its hyperbolic Bartnik mass. We therefore obtain the following corollary.

\begin{coro}\label{coro-main1}
	Let $(\Sigma\cong\bS^2,g_o,H_o)$ be Bartnik data as in Theorem \ref{thm-extension1}. The hyperbolic Barnik mass of $(\Sigma,g_o,H_o)$ satisfies
	\begin{align}
		\m_{B}^{AH}(\Sigma,g_o,H_o)\leq \( \xi + 1 \) \mhh(\S,g_o,H_o)  
+ \dfrac{r_o^3}{2}\xi\left[ \dfrac{H_o^2 }{4}  +\( \xi + 1 \)\(  \xi+ 2 \)\right], 
	\end{align} 
	where $\xi$ is defined as in Theorem \ref{thm-extension1}.
\end{coro}
\vspace{2ex}

\paragraph*{\emph{\underline{Case $b>0$}}}
Recall that we again choose $m<0$. As described above, we will be interested in considering collars where $|m|$ is very large and $k$ is given as a function of $m$ and $b$ via \eqref{eq-kmHpos}, where $b$ will be coupled to $b$ via a small constant $\delta>0$, see below. In this case, we estimate $u_{m,b}$ by noting that
\begin{align*}
b^{-1/2}\frac{d}{ds}\sinh^{-1}(\sqrt b\,\umb (s))
&= \sqrt{\frac{1+b\,\umb(s) ^2-\frac{2m}{\umb(s) }}{1+b\,\umb(s) ^2}},  
\end{align*}
hence
\begin{align*}
\frac{d}{ds}\sinh^{-1}(\sqrt b\,\umb (s))
&\leq \sqrt b\sqrt{1-\frac{2m}{r_o(1+br_o^2)}},
\end{align*}
which in turn implies
\begin{align}\label{eq-m-neg-u-bound}
\umb(s) \leq b^{-1/2} \sinh\( \sqrt b\sqrt{1-\frac{2m}{r_o(1+br_o^2)}}\;s+\sinh^{-1}(\sqrt{b}\,r_o) \).
\end{align}
Moreover, using hyperbolic identities we have
\begin{align} \label{hyp-estimate}
\begin{split}
\umb(s)\leq\,& r_o\cosh\( \sqrt{b}\sqrt{1-\frac{2m}{r_o(1+br_o^2)}}\;s\)\\
&+\sqrt{b^{-1}+r_o^2}\,\sinh\(\sqrt{b}\sqrt{1-\frac{2m}{r_o(1+br_o^2)}}\;s \).
\end{split}
\end{align}

Let $\veps > 0$ and pick any $0 < \delta < \min\left\{ \frac{H_o^2}{4}, \veps, \frac12\right\}$. For any $m<0$ we then have a unique $b=b(m)$ satisfying $k^2b = \delta$, where
	\begin{align*}
	k^2 = \dfrac{H_o^2 r_o^2}{4}\(1 + br_o^2 - \dfrac{2m}{r_o}  \)^{ -1}
	\end{align*}
	is given by \eqref{eq-kmHpos}. That is, we choose
	\begin{align}\label{eq-bdefn}
	b=\delta\(1-\frac{2m}{r_o} \)\(\frac{H_o^2}{4}-\delta \)^{-1}r_o^{-2},
	\end{align}
	which is positive by our smallness assumption on $\delta$. From \eqref{eq-bdefn}, we are also able to express $k$ in terms of $m$ as
\begin{align*}
k^2 = \(\dfrac{H_o^2}{4}-\delta\)r_o^2 \( 1  -\dfrac{2m}{r_o}  \)^{ -1}
\end{align*}
so that, by direct computations,
\begin{align}
\begin{split}\label{extraidentity}
\lim_{m\to-\infty}k^2\,m&=- \(\frac{H_o^2}{4}-\delta\)\frac{r_o^3}{2},\\[0.5ex]
\lim_{m\to-\infty}-\frac{2m}{b r_o^3}&=\frac{1}{\delta}\(\frac{H_o^2}{4}-\delta \),\\[1.25ex]
\lim_{m\to-\infty}b&=+\infty.
\end{split}
\end{align}
From this, it is clear that we can ensure $k^2<\veps$ by choosing $|m|$ sufficiently large.

Now set
\begin{align}\label{Ao-mb}
A_o \definedas \begin{cases} \sqrt{\dfrac{\alpha}{6}}& \beta >0, \\[3ex]
\sqrt{\dfrac{\alpha}{3 + \frac{\beta}{r_o^2}}}\quad& -3r_o^2 < \beta \leq 0. \end{cases}
\end{align}

We will now estimate $u_{m,b}(Ak)$ for a fixed $m<0$ that will be chosen with $\vert m\vert$ large in terms of $\veps$, with the above choices of $k$ and $b$ (which are independent of $A$). For any $A>A_o$ and $\vert m\vert$ large enough in terms of $\veps$, the limits given above show that we can ensure
\begin{align}
\begin{split}\label{umb-eps-est}
\umb(Ak) \leq\,& r_o\cosh\( \sqrt{\delta}A\sqrt{1-\frac{2m}{r_o(1+br_o^2)}}\)\\&+\sqrt{b^{-1}+r_o^2}\,\sinh\(\sqrt{\delta}A\sqrt{1-\frac{2m}{r_o(1+br_o^2)}} \)\\
\leq\,& (r_o + \veps)\exp\(\frac{A}{2}(H_o+\veps) \).
\end{split}
\end{align}

Let us now consider the scalar curvature of the collar. First recall that by \eqref{estimateScal-mb}, the scalar curvature satisfies
\begin{align*}
&R(\gamma) + 6 \geq 2u_{m,b}(Akt)^{-2}\left[ \b -k^2 +\(3-3k^2b - \dfrac{\a}{2 A^2}  \) \umb(Akt)^2\right ].
\end{align*}
In the case where $\beta>0$, by \eqref{estimateScal-mb} and the definition of $A_o$ in \eqref{Ao-mb}, it is clear that choosing $|m|$ sufficiently large ensures $R(\gamma)+6>0$. Similarly, if $-3r_o^2 < \beta \leq 0$, we have from the definition of $A_o$ in \eqref{Ao-mb}
\begin{align*}
R(\gamma) + 6
&\geq 2u_{m,b}(Akt)^{-2}\left[ \b-k^2 +\(3 - \dfrac{\a}{2 A_o^2}  \) \umb(Akt)^2\right ]-6\delta \\
&\geq 2u_{m,b}(Akt)^{-2}\left[ \b -k^2 +\(3 - \frac32 - \frac{\beta}{2r_o^2}  \) r_o^2\right ]-6\delta \\
&\geq 2u_{m,b}(Akt)^{-2}\left[ \frac12(\b+3r_o^2)-k^2\right]-6\delta\\
&\geq \frac{\b+3r_o^2}{u_{m,b}(Ak)^2}-O(\veps),
\end{align*}
so again we obtain $R(\gamma)+6>0$ from \eqref{umb-eps-est} and \eqref{extraidentity} by choosing $|m|$ large enough. It follows from \eqref{umb-eps-est} that for sufficiently large $|m|$ and $A$ sufficiently close to $A_o$ in terms of $\veps$, we can estimate $\mhh(\S_1)$ expressed as in \eqref{eq-HawkingHpos}, and using \eqref{extraidentity} by
\begin{align*}
\mhh(\S_1) &=\dfrac{u_{m,b}(Ak) }{2}\(1- k^2 +(1 -k^2 b)u_{m,b}^2(Ak)  \) + k^2m\\
&\leq  \dfrac{(r_o + \veps)}{2}\exp\(\frac{A}{2}(H_o+\veps) \)\(1 -k^2 +(1 -\delta)  (r_o + \veps)^2\exp\(A(H_o+\veps) \)  \) \\
&\qquad - \(\frac{H_o^2}{4}-\delta \)\frac{r_o^3}{2} + \veps\\
&\leq  \dfrac{r_o}{2} \exp\(\frac{AH_o}{2} \)\(1+(1 -\delta) r_o^2\exp\(AH_o \)  \) - \(\frac{H_o^2}{4}-\delta \)\frac{r_o^3}{2}  + O(\veps) \\
&\leq  \dfrac{r_o}{2} \exp\(\frac{A_oH_o}{2} \)\(1+r_o^2\exp\(A_oH_o \)  \) - \frac{H_o^2 r_o^3}{8} + O(\veps).
\end{align*}

For the sake of presentation, we now define
\begin{align*}
\zeta \definedas \exp\(\frac{A_o H_o}{2} \).
\end{align*}
 
Recall that
\begin{align*} 
\mhh(\S_0)=\frac{r_o}{2}\( 1 -\frac{H_o^2 r_o^2}{4} + r_o^2 \)=\frac{r_o}{2} - \frac{H_o^2 r_o^3}{8} + \frac{r_o^3}{2}.
\end{align*}
By the above, we have shown that for any $\veps>0$ we can construct collars with non-negative scalar curvature of the form \eqref{eqCollar}, satisfying
\begin{align*}
\mhh(\S_1) &\leq
\zeta\mhh(\S_0) + \(1-\zeta\) \frac{H_o^2 r_o^3}{8}+ \zeta\( \zeta^2 - 1 \)\dfrac{r_o^3}{2} +O(\veps).
\end{align*}
Unlike the collars obtained above when $b=0$, these collars require no additional restrictions on the initial Bartnik data. However, they suffer from the fact that the mass increases exponentially in $A_oH_o$ rather than linearly.\\

It is clear that Proposition~\ref{prop-gluing2} can be applied directly, after change of variables $s \definedas Akt$ provided
\begin{align*}
\mhh(\S_1) \geq  -u_{m,b}(Ak)^3.
\end{align*}
By \eqref{eq-HawkingHpos}, this is equivalent to the condition
\begin{align*}
\dfrac{r_o}{2}\(1- k^2 + u_{m,b}(Ak)^2(3 - k^2 b)\)  + k^2 m \geq 0.
\end{align*}

It is straightforward to check that, by again taking $|m|$ sufficiently large and $A$ sufficiently close to $A_o$, this is satisfied if we enforce
\begin{align*}
\frac{H_o^2r_o^2}{4}<1+3r_o^2.
\end{align*}
That is, Proposition \ref{prop-gluing2} again gives us an admissible asymptotically hyperbolic extension from given Bartnik data with controlled mass. Thus, we have established the following theorem.\\

\begin{thm}\label{thm-extension2}
Let $(\S\cong\bS^2,g_o,H_o)$ be Bartnik data such that the Gaussian curvature of $g_o$ satisfies $K(g_o) > -3$ and $H_o$ is a positive constant, and let $r_o$, $\alpha$, and $\beta$ be the constants defined by \eqref{alpha-beta-def}. If the data satisfies 
	\begin{align}\label{eq-cond-thm2}
	\frac{H_o^2r_o^2}{4}<1+3r_o^2,
	\end{align}
then for any 
\begin{align*}
m > m_*,
\end{align*}
	where $m_*$ is defined by
\begin{align*}
	m_*\definedas \zeta\mhh(\S,g_o,H_o)+(\zeta-1)\frac{H_o^2r_o^3}{8}+ \zeta\( \zeta^2 - 1 \)\dfrac{r_o^3}{2}
\end{align*}
with
\begin{align*}
\zeta \definedas \begin{cases} \exp\(\dfrac{H_o}{2}\sqrt{\dfrac{\alpha}{6}}\)& \beta >0, \\[3ex]
	\exp\(\dfrac{H_o}{2}\sqrt{\dfrac{\alpha}{3 + \frac{\beta}{r_o^2}}}\)\quad& -3r_o^2 < \beta \leq 0,
\end{cases}
\phantom{a}\\
\end{align*}
there is an asymptotically hyperbolic Riemannian manifold $(M,\gamma)$ with $R(\gamma) \geq -6$  such that\\[-1.5ex]
\begin{enumerate}[(i)]\itemsep1.5ex
\item the boundary $\partial M$ is isometric to $(\S,g_o)$ and has constant mean curvature $H_o$,
\item outside a compact set, $M$ coincides with the spatial AdS-Schwarzschild manifold of mass $m$, and
\item $M$ is foliated by mean convex spheres that eventually coincide with the coordinate spheres in the spatial AdS-Schwarzschild manifold.\\[-1.5ex]
\end{enumerate}
\end{thm}

Note that \eqref{eq-cond-thm2} is equivalent to 
\begin{align*} 
\mhh(\S,g_o,H_o)>-\(\frac{|\S|_{g_o}}{4\pi} \)^{\frac32}.
\end{align*}

As before, this leads us to another estimate for the hyperbolic Bartnik mass that is valid even for Bartnik data that is very far from round.

\begin{coro}\label{coro-main2}
	Let $(\S\cong\bS^2,g_o,H_o)$ be Bartnik data as in Theorem \ref{thm-extension2}. The hyperbolic Bartnik mass of  $(\S\cong\bS^2,g_o,H_o)$  satisfies
	\begin{align*}
	\m_B^{AH}(\S,g_o,H_o)\leq \zeta\mhh(\S,g_o,H_o)+(\zeta-1)\frac{H_o^2r_o^3}{8}+ \zeta\( \zeta^2 - 1 \)\dfrac{r_o^3}{2} ,
	\end{align*}
where $\zeta$ is defined as in Theorem \ref{thm-extension2}.
\end{coro}
\newpage

\bibliographystyle{amsplain}
\bibliography{AH-V2}
\vfill
\end{document}